\documentclass[article,11pt]{amsart}

\usepackage{amsmath, amsthm, amsfonts,stmaryrd,amssymb}

\usepackage{mathrsfs}                      
\usepackage{mathtools}                     
\usepackage{booktabs}
\usepackage{todonotes}
\usepackage[all]{xy}
\usepackage[utf8]{inputenc}
\usepackage{mathtools}                     
\usepackage{todonotes} 
\usepackage{bbm}                         
\usepackage{framed}
\usepackage{graphicx}   
\usepackage[margin=3.5cm]{geometry}
\usepackage{hyperref} 
\usepackage{xcolor}

\hypersetup{
    colorlinks=true,
    linkcolor=blue,  
    urlcolor=cyan,
    pdftitle={},
    pdfauthor={},
} 
\usepackage{theoremref}
\usepackage{subfigure}

\numberwithin{equation}{section}
\theoremstyle{plain}
\newtheorem{theorem}{Theorem}[section]

\newtheorem{remark}[theorem]{Remark}
\newtheorem{lemma}[theorem]{Lemma}

\newtheorem{proposition}[theorem]{Proposition}

\theoremstyle{definition}
\newtheorem{definition}[theorem]{Definition}

\def\flot{X}
 
\def\ed{\mathrm{d}}

\def\exp{\operatorname{exp}}

\usepackage{accents}

\newcommand{\barr}[1]{#1}

\usepackage{scalerel,stackengine}
\stackMath
\newcommand\rwidehat[1]{%
\savestack{\tmpbox}{\stretchto{%
  \scaleto{%
    \scalerel*[\widthof{\ensuremath{#1}}]{\kern-.6pt\bigwedge\kern-.6pt}%
    {\rule[-\textheight/2]{1ex}{\textheight}}
  }{\textheight}%
}{0.5ex}}%
\stackon[1pt]{#1}{\tmpbox}%
}
\newcommand{\mres}{\mathbin{\vrule height 1.6ex depth 0pt width
0.13ex\vrule height 0.13ex depth 0pt width 1.3ex}}

\let\mc=\mathcal

\let\mr=\mathrm

\begin{document}
\date{\today}
\title[Convergence of a Lagrangian discretization for barotropic fluids]{Convergence of a Lagrangian discretization for barotropic fluids and pourous media flow}
\author[T. O. Gallou\"et]{Thomas O. Gallou\"et}
\address{Thomas O. Gallou\"et  (\href{mailto:thomas.gallouet@inria.fr}{\tt thomas.gallouet@inria.fr}), Team Mokaplan, Inria Paris  75012 Paris, CEREMADE, CNRS, UMR 7534, Université Paris-Dauphine, PSL University, 75016 Paris, France} 
\author[Q. Mérigot]{Quentin Mérigot}
\address{Quentin Mérigot (\href{mailto:quentin.merigot@universite-paris-saclay.fr}{\tt quentin.merigot@universite-paris-saclay.fr}), Université Paris-Saclay, CNRS, Laboratoire de mathématiques d’Orsay, 91405, Orsay, France \and Institut universitaire de France (IUF)
} 
\author[A. Natale]{Andrea Natale}
\address{Andrea Natale (\href{mailto:andrea.natale@inria.fr}{\tt andrea.natale@inria.fr}), Inria, Univ. Lille, CNRS, UMR 8524 - Laboratoire Paul Painlevé, F-59000 Lille, France
}

\maketitle
\renewcommand\thesubfigure{}

\begin{abstract}

 When expressed in Lagrangian variables, the equations of motion for compressible (barotropic) fluids have the structure of a classical Hamiltonian system in which the potential energy is given by the internal energy of the fluid. The dissipative counterpart of such a system coincides with the porous medium equation, which can be cast in the form of a gradient flow for the same internal energy. 
Motivated by these related variational structures, we propose a particle method for both problems in which the internal energy is replaced by its Moreau-Yosida regularization in the $L^2$ sense, which can be efficiently computed as a semi-discrete optimal transport problem. Using a modulated energy argument which exploits the convexity of the problem in Eulerian variables, we prove quantitative convergence estimates towards smooth solutions. We verify such estimates by means of several numerical tests.
\end{abstract}

\section{Introduction}
The Euler equations describing the evolution of a barotropic fluid in a compact domain $M\subset \mathbb{R}^d$ with Lipschitz boundary and on a time interval $[0,T]$ are given by the following system of equations:
\begin{equation}\label{eq:euler}
\left\{
\begin{array}{l}
\partial_t (\rho u) + \nabla \cdot (\rho u \otimes u) +\nabla P (\rho) =0 \,,\\
\partial_t \rho + \mathrm{div} (\rho u) = 0 \,,
\end{array}\right.
\end{equation}
where $\rho(t,x)\geq 0 $ is the fluid density, $u(t,x)\in\mathbb{R}^d$ is the Eulerian velocity and the function $P: [0,\infty) \rightarrow \mathbb{R}$ defines the pressure as a function of the density. The first equation in \eqref{eq:euler} is generally referred to as the momentum equation, whereas the second is the continuity equation and describes local mass conservation in the fluid. The system is supplemented by the initial and boundary conditions:
\[
\rho(0, \cdot ) = \rho_0 \,, \quad u(0,\cdot) = u_0 \,, \quad u \cdot n_{\partial M} =0\, \text{ on }[0,T]\times \partial M \,,
\]
where $n_{\partial M}$ is the outward normal to the boundary $\partial M$. Smooth solutions conserve the total energy
\begin{equation}\label{eq:energytot}
\int_M \frac{1}{2}  |u|^2 \rho \ed x + \int_M U(\rho)\, \ed x\,,
\end{equation}
where $U: [0,\infty) \rightarrow \mathbb{R}$ is a smooth strictly convex function, superlinear at infinity, defining the internal energy of the fluid. This is related to the pressure by the thermodynamic relations
\begin{equation}\label{eq:thermodynamic}
P(r) = r U'(r) - U(r) \,, \quad P'(r) = r U''(r).
\end{equation}
Different choices of the internal energy $U$ lead to different models. The two most classical examples are: 
\begin{enumerate}
\item polytropic fluids, which correspond to $U(r) = r^m/(m-1)$ with $m>1$, and $P(r) = r^m$ (these include isentropic fluids, and the Saint-Venant system modelling gravity driven shallow water flows for $m=2$);
\item isothermal fluids, which correspond to $U(r) = r\log(r) -r$ and $P(r)=r$.
\end{enumerate}

Adding a friction term $-\zeta \rho u$ on the right-hand side of the momentum equation, i.e. the first equation in the system \eqref{eq:euler}, and considering the high friction limit $\zeta \rightarrow \infty$, one formally obtains $u = -\nabla U'(\rho)$, which substituted into the continuity equation yields 
\begin{equation}\label{eq:porous}
\partial_t \rho - \Delta P(\rho) =0\,.
\end{equation}
In particular, the choice  $U(r) = r^m/(m-1)$ with $m>1$ and $P(r)=r^m$, which is associated with polytropic fluids, yields the porous medium equation. Similarly, the choice $U(r) = r \log r-r$ and $P(r) =r$ corresponding to isothermal fluids, yields the heat equation.

\subsection{Lagrangian formulation} 


For both the compressible Euler system \eqref{eq:euler} and its high friction limit \eqref{eq:porous}, the density evolves according to the continuity equation with respect to a time-dependent vector field $u$.  Let $S_0 \subseteq M$ be the support of the initial density $\rho_0$ and $\flot: [0,T] \times S_0 \rightarrow M$ be the flow associated with $u$, i.e.\ the time-dependent map satisfying the flow equation 
\begin{equation}\label{eq:flow} \dot{\flot}_t = u(t,{\flot}_t) \end{equation}
with initial condition $\flot_0 = \mathrm{Id}|_{S_0}$, where $\mathrm{Id}$ is the identity map on $\mathbb{R}^d$. If $\rho_0$ and $u$ are sufficiently regular, then the flow equation \eqref{eq:flow} and the continuity equation have both a unique strong solution, and the density is the pushforward of $\rho_0$ by the flow, i.e.\ $\rho(t,\cdot)=\flot_{t\#} \rho_0$, where the pushforward is defined by the condition 
\begin{equation}\label{eq:density}
(X_{t\#} \rho_0)[B] =  \rho_0 [\flot_t^{-1}(B)] \quad  \textrm{  for any }  B \subset  M.
\end{equation} 
In general, equation \eqref{eq:density} defines $X_{t\#} \rho_0$ only as a measure on $M$. However, if $X_t$ is a smooth invertible map, $X_{t\#} \rho_0$ is absolutely continuous with respect to the Lebesgue measure $\ed x$, and we identify it with its smooth density.

Using equation \eqref{eq:flow} and \eqref{eq:density}, the total energy of the fluid \eqref{eq:energytot} can then be written in terms of $\flot$ only as follows:
\begin{equation}\label{eq:energyfluid}
\int_M \frac{1}{2}  |\dot{\flot}_t|^2 \rho_0 \ed x + \int_M U(\flot_{t\#} \rho_0)\, \ed x\,.
\end{equation}
Let $\mathbb{X}\coloneqq L^2_{\rho_0}(S_0;\mathbb{R}^d)$. In the smooth setting, we can interpret the energy \eqref{eq:energyfluid} as a functional on curves of smooth invertible maps in $C^\infty(S_0;M)$, viewed as a manifold in $\mathbb{X}$ with the induced metric. The associated Euler-Lagrange equations coincide with Newton's second law:
\begin{equation}\label{eq:secondorder}
\ddot{\flot}_t = - \nabla_{\mathbb{X}} \mathcal{F}(\flot_t)\,, \quad \mathcal{F}(\sigma) \coloneqq \int_M U(\sigma_{\#} \rho_0)\, \ed x\,,
\end{equation}
where we identify the gradient  $\nabla_{\mathbb{X}} \mathcal{F}(\flot_t)$ with an element of $\mathbb{X}$ (see Remark \ref{Re:caclulgradientformel} for a formal computation of $\nabla_{\mathbb{X}} \mathcal{F}(\flot_t)$). Equation \eqref{eq:secondorder} is the Lagrangian equivalent to the momentum equation in \eqref{eq:euler}, and in particular from its solutions one can retrive the solutions to the Euler system \eqref{eq:euler} using the flow equation \eqref{eq:flow} and the definition of pushforward  \eqref{eq:density}.

In the case of the high friction limit \eqref{eq:porous}, the flow evolves according to a gradient flow dynamics, which correspond to the equation:
\begin{equation}\label{eq:gradient}
\dot{\flot}_t =  -\nabla_{\mathbb{X}} \mathcal{F}(\flot_t)\,. 
\end{equation}
Here, equation \eqref{eq:gradient} is equivalent to the condition $ u =- \nabla U'(\rho)$, and from its solutions one can retrive the solutions to \eqref{eq:porous} by pushforward of the initial density as in \eqref{eq:density}. 


The point of view described above for the compressible Euler system is one of the possible generalizations of the approach developed by Arnold for the incompressible Euler equations (see, e.g., Proposition 2.7 in \cite{khesin2018geometric}), which he intrepreted as the geodesic equation on the group of volume-preserving diffeomorphisms with the $L^2$ metric \cite{arnold1966geometrie}. 
On the other hand, the gradient flow structure in \eqref{eq:gradient} is the Lagrangian counterpart of the Wasserstein gradient flow interpretation of equation \eqref{eq:porous}, developed in the celebrated works of Otto \cite{otto2001geometry} and Jordan, Kinderlehrer, and Otto \cite{jordan1998variational}.

 In this paper, we will construct discrete versions of the systems \eqref{eq:secondorder} and \eqref{eq:gradient} in which the flow is approximated by a curve of (non-smooth and non-injective) maps belonging to a finite-dimensional subpace of $\mathbb{X}$. As a consequence of this extrinsic point of view, we will  regard the internal energy $\mathcal{F}$ in equation \eqref{eq:secondorder} as a real-valued functional on the whole space $\mathbb{X}$, which we set to $+\infty$ when $\sigma_{\#} \rho_0$ is not absolutely continuous with respect to the Lebesgue measure $\ed x$ restricted to $M$.

\subsection{Space discretization}\label{sec:spacedisc}
We now turn to the design of the Lagrangian scheme, i.e.\ an evolutive system for a finite number of particles, to approximate both Euler and gradient flows. In order to define the evolution of the particles we introduce a discrete equivalent of the Lagrangian variational structure highlighted in the previous section. This also allows us to preserve at the discrete level the link between the two models described above.

Let $N\in \mathbb{N}^*$ and consider a partition $\mathcal{P}_N:=(P_i)_{1\leq i \leq N}$ of the initial support $S_0 \subseteq M$ in $N$ regions with $h_N \coloneqq \max_i \mathrm{diam}(P_i) \lesssim N^{-d}$. We define $\mathbb{X}_N \subset \mathbb{X}$ as the space of functions that are constant on each subdomain $P_i$, i.e.\
\[
\mathbb{X}_N \coloneqq \{X_N \in \mathbb{X} ~|~ X_N(\omega) = X_N^i \in \mathbb{R}^d ~ \text{ for a.e. } \omega \in P_i,~ 1\leq i \leq N\}\,.
\]
Then, we discretize the flow $\flot$ by a curve $\flot_N: [0,T]\rightarrow \mathbb{X}_N$, and for any $t\in[0,T]$ we identify $\flot_N(t)$ with the vector of the position of the particles $(\flot_N^i(t))_i\in \mathbb{R}^{dN}$ where $\flot_N^i(t)\in\mathbb{R}^d$ is the image of any point in $P_i$ by the map $\flot_N(t)$ and therefore carries a mass $\rho_0[P_i]$. As in the continuous case the density of the fluid is given by the pushforward  $\rho_N(t) = \flot_N(t)_\# \rho_0$, or more explicitly by the sum of all the particles weighted by their respective masses: 
\begin{equation}\label{eq:densityN}
\rho_N(t) = \sum_{i=1}^N \rho_0[P_i] \delta_{\flot_N^i(t)}\,.
\end{equation}

Since $\rho_N(t)$ is not absolutely continuous, the internal energy $\mathcal{F}$ is identically $+\infty$ on all of $\mathbb{X}_N$, and in order to define our numerical approximation, we need to replace it by a regularized version. In this paper we consider the Moreau-Yosida regularization of $\mathcal{F}$, which is given by
\begin{equation}\label{eq:my}
\mathcal{F}_\varepsilon (\flot) \coloneqq \inf_{\sigma \in \mathbb{X}} \frac{\|\flot - \sigma\|^2_{\mathbb{X}} }{2\varepsilon} + \mathcal{F}(\sigma)\,,
\end{equation}
for any $\flot \in \mathbb{X}$ and for a fixed $\varepsilon>0$. Note that problem \eqref{eq:my} always admits minimizers when  $\flot \in \mathbb{X}_N$, but these are in general not unique. 

In order to mimic the continuous case, the discrete dynamics is thus given by the Euler (resp. gradient) flow of $\mathcal{F}_\varepsilon$ in $\left(\mathbb{X}_N,  L^2_{\rho_0} \right)$. More precisely, the space discretization of the Euler system \eqref{eq:euler} reads as follows:
\begin{equation}\label{eq:disceuler}
\ddot{\flot}_N(t) = - P_{\mathbb{X}_N} \nabla_\mathbb{X} \mathcal{F}_\varepsilon(\flot_N(t)) \,, \quad \flot_N(0) = \mathrm{Id}_N \,, \quad \dot{\flot}_N(0) = u_0 \circ \mathrm{Id}_N \,,
\end{equation}
where $P_{\mathbb{X}_N}$ is the $L^2_{\rho_0}$ projection onto $\mathbb{X}_N$, and we set $ \mr{Id}_N \coloneqq P_{\mathbb{X}_N}\mathrm{Id}|_{S_0}$. Note that the left-hand side of equation \eqref{eq:disceuler} can be identified with the vector collecting the acceleration of the particles $(\ddot{\flot}_N^i(t))_i \in \mathbb{R}^{dN}$. The right-hand side is just the gradient of $\mathcal{F}_\varepsilon$ viewed as a function on $\mathbb{X}_N$, and it is uniquely defined for almost every point in $\mathbb{X}_N$ (see Proposition \ref{prop:gradient} for a precise statement). In particular, we have 
\begin{equation}
P_{\mathbb{X}_N} \nabla_\mathbb{X} \mathcal{F}_\varepsilon(\flot_N) = \frac{\flot_N - P_{\mathbb{X}_N} \flot^\varepsilon_N}{\varepsilon} \,,\quad \flot_N^\varepsilon \in \underset{\sigma \in \mathbb{X}}{\mathrm{argmin}} \, \frac{\|\flot_N - \sigma\|^2_{\mathbb{X}} }{2\varepsilon} + \mathcal{F}(\sigma)\,,
\end{equation}
for almost any $\flot_N \in \mathbb{X}_N$. As in the continuous setting,  the total energy of the system at time $t$ is given by the sum of the kinetic and internal energy, where we replace now the internal energy by its regularized version:
\begin{equation}\label{eq:totenergy}
 \mc{E}_\varepsilon(t,\flot_N) \coloneqq \sum_{i=1}^N \frac{1}{2} |\dot{\flot}_N^i(t)|^2 \rho_0[P_i] + \mathcal{F}_\varepsilon(\flot_N(t))\,,
\end{equation}
and this is conserved by smooth solutions of \eqref{eq:disceuler}.

Similarly, the discrete version of the gradient flow \eqref{eq:gradient} is given by
\begin{equation}\label{eq:discgradient}
\dot{\flot}_N(t) = - P_{\mathbb{X}_N} \nabla_\mathbb{X} \mathcal{F}_\varepsilon(\flot_N(t))\,, \quad \flot_N(0) = \mr{Id}_N\,.
\end{equation}
Here, the total energy at time $t$ is simply given by the internal energy $\mathcal{F}_\varepsilon(\flot_N(t))$, and it is dissipated by smooth solutions of \eqref{eq:discgradient}. 

\subsection{Time discretization} \label{sec:timediscrete}
The variational structure of the space-discrete systems described so far can be exploited to design a stable time discretization. The method we describe here consists in considering different approximations of the energy in each time step, and is modelled on the strategy proposed by Brenier in \cite{brenier2000derivation}. 


Let $\tau>0$ a fixed time step, $N_T \in \mathbb{N}^*$ be the number of time steps with $T = \tau N_T$, and $t_n \coloneqq n \tau$ for any $0\leq n \leq N_T$. We define a discrete-time approximation of system \eqref{eq:disceuler}, by considering the $C^1$ curves $\flot_N: [0, T ]\mapsto \mathbb{X}_N$ satisfying in each time interval $[t_n, t_{n+1})$ the equation 
\begin{equation}\label{eq:disceulertime}
\ddot{\flot}_N(t) = - \frac{\flot_N(t) - P_{\mathbb{X}_N} \flot_N^\varepsilon(t_n)}{\varepsilon} \,, 
\end{equation}
where 
\begin{equation}\label{eq:xneps}
 \flot_N^\varepsilon(t_n) \in \underset{\sigma \in \mathbb{X}}{\mathrm{argmin}} \, \frac{\|\flot_N(t_n)  - \sigma\|^2_{\mathbb{X}} }{2\varepsilon} + \mathcal{F}(\sigma)\,,
\end{equation}
and with the same initial condition as in \eqref{eq:disceuler}. This system is conservative in each interval $[t_n,t_{n+1})$ for the energy 
\begin{equation}\label{eq:totenergyd}
\mathcal{E}_{\varepsilon}^n(t,\flot_N) \coloneqq \sum_{i=1}^N \frac{1}{2} |\dot{\flot}_N^i(t)|^2 \rho_0[P_i] + \frac{\|\flot_N(t)  - \flot_N^\varepsilon(t_n)\|^2_{\mathbb{X}} }{2\varepsilon} + \mathcal{F}(\flot_N^\varepsilon(t_n))\,.
\end{equation}
The total energy $\mathcal{E}_\varepsilon(t,\flot_N)$ defined in \eqref{eq:totenergy} is however dissipated in general since, by definition of the regularized energy $\mathcal{F}_\varepsilon$, we have
\begin{equation}\label{eq:energystab}
\mathcal{E}_{\varepsilon}(t_{n+1},\flot_N) \leq \mathcal{E}_{\varepsilon,\tau}^n(t_{n+1},\flot_N) = \mathcal{E}_{\varepsilon,\tau}^n(t_{n},\flot_N) = \mathcal{E}_{\varepsilon}(t_{n},\flot_N)\,.
\end{equation}

The discrete-time approximation of the gradient flow \eqref{eq:discgradient} is given by a continuous curve  $\flot_N: [0, T ]\mapsto \mathbb{X}_N$ which on each interval $[t_n,t_{n+1})$ is the gradient flow on $\mathbb{X}_N$ for the energy:
\begin{equation}\label{eq:regpotential}
 \frac{\|\flot_N(t)  - \flot_N^\varepsilon(t_n)\|^2_{\mathbb{X}} }{2\varepsilon} + \mathcal{F}(\flot_N^\varepsilon(t_n))\,.
\end{equation}
More explicitly, a discrete solution is any  $C^0$ curve $\flot_N: [0, T ]\mapsto \mathbb{X}_N$  which satisfies in each time interval $[t_n,t_{n+1})$,
\begin{equation}\label{eq:discgradienttime}
\dot{\flot}_N(t) = - \frac{\flot_N(t) - P_{\mathbb{X}_N} \flot_N^\varepsilon(t_n)}{\varepsilon} \,,
\end{equation}
with  $\flot_N^\varepsilon(t_n)$ defined as in \eqref{eq:xneps}, 
and the same initial condition as in \eqref{eq:discgradient}. Also in this case the internal energy $\mathcal{F}_\varepsilon(\flot_N(t))$ is dissipated along the evolution, since we have
\begin{equation}\label{eq:energystabgrad}
\mathcal{F}_\varepsilon(\flot_N(t_{n+1})) \leq   \frac{\|\flot_N(t_{n+1})  - \flot_N^\varepsilon(t_n)\|^2_{\mathbb{X}} }{2\varepsilon} + \mathcal{F}(\flot_N^\varepsilon(t_n)) \leq \mathcal{F}_\varepsilon(\flot_N(t_{n}))\,.
\end{equation}

\subsection{Relation with previous works and convergence results}

Using a Lagrangian formulation for the discretization of problems \eqref{eq:euler} and \eqref{eq:porous} enables us to reproduce the conservative and gradient flow structure of the corresponding models. In turn, this allows us to construct stable numerical methods as in \eqref{eq:disceulertime} and \eqref{eq:discgradienttime} to discretize their solutions. 
Similar strategies were already explored in the 1990s, during the emergence of particle methods, for example in the context of the discretization of the incompressible Euler equations in the works of Buttke \cite{buttke1993velicity} and Russo \cite{russo1998impulse}. 
Such methods can be seen as instances of the more general Smoothed Particle Hydrodynamics (SPH) discretizations, where the interaction forces amongst the particles are computed by reconstructing the fluid density through convolution with a fixed kernel (see, e.g., the review articles \cite{lind2020review,monaghan2005smoothed} and references therein), and which have been widely used in the context of the discretization of fluid models. 

Recent SPH methods explicitely exploit the variational structure of the models for the construction of the method itself as in \cite{evers2018continuum}. In the same article, the Authors also established a general (non-quantitative) convergence result towards measure-valued solutions of problem \eqref{eq:euler} for its discretization in space only. In another recent work \cite{franz2018convergence}, the Authors proved quantitative convergence estimates with modulated energy techniques but limited to the case $P(r) =r^2$ and for the discretization in space only. This last work also highlights how the choice of the kernel is crucial to obtain convergence.

The discretization strategy we use in this paper is closely related to the one developed by Brenier \cite{brenier2000derivation}, who proposed a discretization of incompressible Euler which replaces the incompressibility constraint by a potential term given by the $L^2$ distance from the set of measure-preserving maps, discretized as permutations of a fixed regular grid. The potential term used by Brenier can be reinterpreted as a Moreau-Yosida regularization (as in \eqref{eq:my}) of an energy given by the convex indicator function of the Lebesgue measure.
Gallou\"et and Mérigot  \cite{gallouet2018lagrangian} later used a similar approach, but rephrased as a particle method, which allowed them to employ efficient semi-discrete optimal transport techniques to compute the discrete solution, and at the same time improved the convergence estimates of \cite{brenier2000derivation} using a modulated energy approach. Note that the use of semi-discrete optimal transport techniques to simulate fluids was first launched by the work of Mérigot and Mirebeau \cite{merigot2016minimal} to solve the geodesic problem associated with the incompressible Euler equations.

Our convergence results generalize the one in \cite{gallouet2018lagrangian} to the compressible and gradient flow setting. Differently from SPH methods, here the density is reconstructed via a Moreau-Yosida regularization (i.e.\ as the push-forward of $\rho_0$ by the regularized flow $\flot_N^\varepsilon$), which eliminates the problem of selecting a kernel, the reconstruction being deeply linked with the energy itself (see Proposition \ref{prop:gradient}). On the other hand, the kernel length-scale parameter of SPH methods is replaced here by the parameter $\varepsilon$ in the regularized functional \eqref{eq:my}. 

The main results of this paper are contained in Theorem \ref{th:convsemieuler} and \ref{th:convsemigradient} below. The central issue of the proofs is the construction of an appropriate modulated energy to measure the discrepancy error between the discrete and continuous solution. In this work we construct a modulated energy exploiting the convexity of the energy in the Eulerian setting, which is lost in the Lagrangian formulation, and the particular structure of the Moreau-Yosida regularization.  
It should be noted that for convex energies, modulated energy estimates of the type we use here are classical tools for the study of problems \eqref{eq:euler} and \eqref{eq:porous} (see, e.g., Chapter 5 in \cite{dafermos2005hyperbolic}): namely, to prove weak-strong stability and uniqueness results, and to establish convergence in the high friction limit from entropy weak solutions of the Euler equations \eqref{eq:euler} with friction to porous media flow \eqref{eq:porous} \cite{lattanzio2013relative}. Note also that such tecnhiques are not limited to the cases we consider in this article, and can be generalized to treat also less regular energies (see, e.g.,  \cite{giesselmann2017relative,lattanzio2017gas}, for a framework covering the Euler-Korteweg and Euler-Poisson theory).

Another important point is related to the time discretization. The method we use in this work, described in Section \ref{sec:timediscrete}, directly derives from that used by Brenier in \cite{brenier2000derivation} for the incompressible Euler equations. It is specially adapted to the structure of the Moreau-Yosida regularization, and consists in devicing a quadratic approximation of the energy (see equation \eqref{eq:regpotential}) which dominates the regularized energy over each time-step. This naturally implies the stability of the discrete solutions (see equations \eqref{eq:energystab} and \eqref{eq:energystabgrad}), which is an essential element for the convergence results below. { Note that symplectic integrators \cite{hairer2006geometric} could be another natural choice for the discretization of the Hamiltonian system \eqref{eq:disceuler}. This choice was explored in \cite{gallouet2018lagrangian} for incompressible Euler, but it is more difficult to analyze due to the lack of an explicit control of the continuous energy of the system}. Another approach which we do not explore in this paper is the time discretization developed in \cite{gangbo2009optimal,cavalletti2014variational} (see also its numerical implementation in \cite{westdickenberg2010variational}) which is better adapted to the non-smooth setting since it is designed to overcome the non-uniqueness issues related to the notion of entropy solutions. 

The convergence estimate we obtain for the discretization of \eqref{eq:euler} is the following:

\begin{theorem}\label{th:convsemieuler}
Suppose that $(\barr{\rho},\barr{u}):[0,T]\times M \rightarrow [0,\infty) \times \mathbb{R}^d$ is a strong solution to \eqref{eq:euler} such that $u \cdot n_{\partial M} = 0$ on $[0,T]\times \partial M$, with  $U:[0,\infty) \rightarrow \mathbb{R}$ being a smooth strictly convex and superlinear function such that \eqref{eq:pbound} holds. Suppose that $u\in C^1([0,T],C^{2,1}(M,\mathbb{R}^d))$, $\rho_0\in C^{1,1}(M)$, and that either  $\barr{\rho}_0\geq \barr{\rho}_{min}>0$ or that $U$ admits a right third derivative at $0$, i.e.\ $|U'''_+(0)|<\infty$.
Suppose in addition that $\flot_N : [0, T ] \rightarrow \mathbb{X}_N$ is a $C^1$ curve which satisfies \eqref{eq:disceulertime} for all times in $[0, T]$, with initial conditions  $\flot_N(0) = \mathrm{Id}_N$ and $\dot{\flot}_N(0) = u(0,\mathrm{Id}_N(\cdot))$. Then, denoting by $\flot$ the  flow associated with $\barr{u}$ satisfying $\flot(0) = \mathrm{Id}|_{S_0}$,
\begin{equation}\label{eq:esteuler}
\sup_{t\in[0,T]} \|\dot{\flot}_N(t) - u(t,\flot_N(t))\|_{\mathbb{X}}^2 +  \| \flot_N(t) - \flot(t) \|_{\mathbb{X}}^2 \leq C( \frac{h_N^2}{\varepsilon} +  h_N +  \varepsilon +  \frac{\tau}{\varepsilon}) \,,
\end{equation}
where $C>0$ depends only on $\sup_{t\in[0,T]} (\|u(t)\|_{C^{2,1}} + \|\partial_t u(t)\|_{C^{2,1}}) $,  $\|\rho_0\|_{C^{1,1}}$, and on $U$, $T$ and $d$. 
\end{theorem}
For what concerns the discretization of dissipative problems of the type \eqref{eq:porous}, several Lagrangian discretizations based on their gradient flow structure \eqref{eq:gradient} have already been developed (see, e.g., the method in \cite{carrillo2019blob} which is close to SPH methods, or in general the review \cite{carrillo2020lagrangian} and references therein). The discretization we consider here has been studied in \cite{leclerc2020lagrangian} (in the time-continuous setting), where the Authors considered more general energies than those we treat here, modelling for example congestion phenomena, and proved the convergence of the discrete measures \eqref{eq:densityN} to solutions of the associated PDE in dimesion one. The result requires an a priori estimate on the regularized flow $X_N^\varepsilon$ which is not proven in higher dimensions. Here we circumvent this issue using the same arguments as in Theorem \ref{th:convsemieuler}, and in particular by a careful choice of a modulated energy and by exploiting the smoothness of the continuous solutions. The convergence estimate we obtain for the discretization of problem \eqref{eq:porous} is the following:

\begin{theorem}\label{th:convsemigradient}
Suppose that $\barr{\rho}:[0,T]\times M \rightarrow [0,\infty)$ is a strong solution to \eqref{eq:gradient} such that $\nabla U'(\rho) \cdot n_{\partial M} = 0$ on $[0,T]\times \partial M$, with $U:[0,\infty) \rightarrow \mathbb{R}$ being a smooth strictly convex and superlinear function such that \eqref{eq:pbound} holds. Suppose that $u\coloneqq-\nabla U'(\rho)$ is of class $C^{2,1}$ in space, uniformly in time, $\rho_0\in C^{1,1}(M)$, and that either  $\barr{\rho}_0\geq \barr{\rho}_{min}>0$ or that $U$ admits a right third derivative at $0$, i.e.\ $|U'''_+(0)|<\infty$.
Suppose in addition that $\flot_N : [0, T ] \rightarrow \mathbb{X}_N$ is a $C^0$ curve which  satisfies \eqref{eq:discgradienttime} for all times in $[0, T]$ with initial conditions $\flot_N(0) = \mathrm{Id}_N$. Then, denoting by $\flot$ the flow associated with $\barr{u}$ satisfying $\flot(0) = \mathrm{Id}|_{S_0}$,
\begin{equation}\label{eq:estgradient}
\sup_{t\in[0,T]}  \int_0^t  \|\dot{\flot}_N(s) - u(s,\flot_N(s))\|_{\mathbb{X}}^2\,\ed s +  \| \flot_N(t) - \flot(t) \|_{\mathbb{X}}^2
\leq C( \frac{h_N^2}{\varepsilon} +  h_N +  \varepsilon +  \frac{\tau}{\varepsilon})\,,
\end{equation}
where $C>0$ depends only on $\sup_{t\in[0,T]} \|\nabla U'(\rho(t))\|_{C^{2,1}}$,  $\|\rho_0\|_{C^{1,1}}$,  and on $U$, $T$ and $d$. 
\end{theorem}

\begin{remark} 
The modulated energy we use to prove the estimates above has an additional term, if one compares it to the left-hand sides of equations \eqref{eq:esteuler} and \eqref{eq:estgradient}, which is associated with the internal energy $\mc{F}$ and which is omitted in order to simplify the statements. This term is discussed in detail in Section \ref{sec:modulated} and actually implies a stronger control on the reconstructed density associated with the regularized flow $\flot_N^\varepsilon$.
\end{remark}

\section{Moreau-Yosida regularization}

In this section we collect some properties of the regularized energy in \eqref{eq:my}. We provide an equivalent Eulerian formulation of such an energy using the $L^2$-Wasserstein distance on the space of positive measures of fixed mass, and we also give a characterization of its gradient in terms of the pressure, which will be useful to prove our convergence results. 

We start by introducing the Eulerian counterpart to the internal energy functional in \eqref{eq:secondorder},
which we obtain by regarding this as a function of the density rather than the Lagrangian flow map. More precisely, denoting by $\mc{M}_+(\mathbb{R}^d)$ the set of positive finite measures on $\mathbb{R}^d$, we define  $\mathcal{U}:\mc{M}_+(\mathbb{R}^d) \rightarrow \mathbb{R}$ as follows:
\begin{equation}\label{eq:potenergy}
\mathcal{U}(\rho)\coloneqq  \left \{
\begin{array}{ll}
 \int_M U(\rho) \, \ed x & \text{if } \rho\ll \ed x  \mres M,\\
 +\infty & \text{otherwise.}
\end{array}\right.
\end{equation}
Then, the functional $\mathcal{F}:\mathbb{X} \rightarrow \mathbb{R}$ in \eqref{eq:secondorder} can be equivalently defined by 
\[\mathcal{F}(\flot) \coloneqq \mathcal{U}(\flot_\# \rho_0).\]

We define  $\mathcal{U}_\varepsilon (\rho):\mc{M}_+(\mathbb{R}^d) \rightarrow \mathbb{R}$ as the Moreau-Yosida regularization of $\mathcal{U}$ with respect to the $L^2$-Wasserstein distance, i.e.\
\begin{equation}\label{eq:mye}
\mathcal{U}_\varepsilon (\rho) \coloneqq \min_{\mu \in\mc{M}_+(\mathbb{R}^d)}\frac{W^2_2(\rho,\mu)}{2\varepsilon} + \mathcal{U}(\mu)\,.
\end{equation}
The quantity $W_2(\rho,\mu)$ is the $L^2$-Wasserstein distance between $\rho$ and $\mu$ (see, e.g., Chapter 5 in \cite{santambrogio2015optimal}), and it can be defined via the following minimization problem:
\[
W_2^2(\rho,\mu)\coloneqq \min_{\substack{\gamma \in \Pi(\rho,\mu)}} \int {|x-y|^2}\,\ed \gamma(x,y) \,,
\]
where $\Pi(\rho,\mu)$ is the set of positive measures on $\mathbb{R}^d\times \mathbb{R}^d$ with marginals $\rho$ and $\mu$, and we set $W^2_2(\rho,\mu)= +\infty$ whenever $\rho$ and $\mu$ have different total mass.
Since $U$ is strictly convex and superlinear, for any $\rho \in\mc{M}_+(\mathbb{R}^d)$ (with finite second moment) the function minimized in problem \eqref{eq:mye} is lower semi-continuous with respect to the Wasserstein metric (see, e.g., Proposition 7.7 in \cite{santambrogio2015optimal}) and therefore it admits a unique minimizer which we denote $\rho^\varepsilon$.
The link between the Eulerian \eqref{eq:mye} and Lagrangian form \eqref{eq:my} of the regularized energy is established in the following lemma.

\begin{lemma}\label{lem:eulerianmy}
Let $\flot_N \in \mathbb{X}_N$ and $\rho_N = (\flot_N)_\# \rho_0$, with $\rho_0 \in \mc{M}_+(\mathbb{R}^d)$ such that $\rho_0\ll \ed x \mres M$. Then, $\mathcal{F}_\varepsilon(\flot_N) = \mathcal{U}_\varepsilon(\rho_N)$. In particular, there exists a convex function $\psi : \mathbb{R}^d \rightarrow  \mathbb{R}$, whose gradient is uniquely defined, such
that $\flot_N^\varepsilon$  is a minimizer associated with $\flot_N$ in problem \eqref{eq:my}, i.e.
\[ 
{\flot}^\varepsilon_N \in \underset{\sigma \in \mathbb{X}}{\mathrm{argmin}} \frac{\|\flot_N-\sigma
\|^2_{\mathbb{X}}}{2\varepsilon}  + \mathcal{F}(\sigma)\,,\] 
if and only if $\flot_N = \nabla \psi \circ \flot_N^\varepsilon$ up to a negligible set.  Moreover, let 
\[
{\rho}^\varepsilon_N \coloneqq \underset{\mu \in\mc{M}_+(\mathbb{R}^d)}{ \mathrm{argmin}}  \frac{W^2_2(\rho_N,\mu)}{2\varepsilon} + \mathcal{U}(\mu)\,.
\]
Then, ${\rho}^\varepsilon_N = ({\flot}^\varepsilon_N )_\# \rho_0$.
\end{lemma}
\begin{proof}
Let $\Pi(\rho_N,\mu)$ the set of positive measures on $\mathbb{R}^d\times \mathbb{R}^d$ with marginals $\rho_N = (\flot_N)_\# \rho_0$ and $\mu$. Since $\rho_0$ is a.c., for any $\mu\in\mc{M}_+(\mathbb{R}^d)$ with the same total mass of $\rho_0$, there exists a $\sigma \in \mathbb{X}$ such that $\sigma_\#\rho_0 = \mu$, and we can construct a measure $(\flot_N,\sigma)_\# \rho_0 \in \Pi(\rho_N,\mu)$. This implies that
\begin{equation}\label{eq:kantorovich}
\min_{\substack{\gamma \in \Pi(\rho_N,\mu)}} \int \frac{|x-y|^2}{2\varepsilon} \,\ed \gamma(x,y) + \mathcal{U}(\mu) \leq
 \frac{\|\flot_N-\sigma\|^2_{\mathbb{X}}}{2\varepsilon} + \mathcal{U}(\sigma_\# \rho_0)  \,.
\end{equation}
Therefore, taking the infimum over $\sigma$ on both sides of \eqref{eq:kantorovich} yields $\mathcal{U}_\varepsilon(\rho_N) \leq \mathcal{F}_\varepsilon(\flot_N)$.

To prove the reverse inequality, consider again $\rho_N = (\flot_N)_\# \rho_0 = \sum_i \rho_0[P_i] \delta_{\flot_N^i}$ and let $\rho^\varepsilon_N$ the associated minimizer of problem \eqref{eq:mye}. By Brenier's theorem \cite{brenier1991polar}, there exists a unique transport map given by the gradient of a convex function $\psi$ such that $(\nabla \psi)_\# \rho^\varepsilon_N = \rho_N$ and $W^2_2(\rho_N,\rho^\varepsilon_N)= \int_M | \nabla \psi -\mathrm{Id} |^2 \ed \rho^\varepsilon_N$ . This coincides with the optimal transport map from $\rho^\varepsilon_N$ to $\rho_N$. For any $1\leq i \leq N$, denote $L_i \coloneqq (\nabla \psi)^{-1}(\flot_N^i)$ so that $\rho^\varepsilon_N[L_i] = \rho_0[P_i]$, and let $\sigma_i:P_i\rightarrow L_i$ be any map such that $(\sigma_i)_\# \rho_0|_{P_i} = \rho^\varepsilon_N|_{L_i}$. Then we can take $\flot_N^\varepsilon \in \mathbb{X}$ to be the map defined by $\flot_N^\varepsilon|_{P_i} = \sigma_i$.
Clearly, $\flot_N= \nabla \psi \circ \flot_N^\varepsilon$ by construction and
\[
\mathcal{F}_\varepsilon(\flot_N) \leq \frac{\|\flot_N-\flot_N^\varepsilon \|^2_{\mathbb{X}}}{2\varepsilon} + \mathcal{U}((\flot_N^\varepsilon)_\# \rho_0)   = \int_M \frac{|\nabla \psi -\mathrm{Id} |^2}{2\varepsilon} \ed \rho^\varepsilon_N +\mathcal{U}(\rho^\varepsilon_N)  = \mathcal{U}_\varepsilon(\rho_N)\,.
\]
Therefore, we have the equality $\mathcal{U}_\varepsilon(\rho_N) = \mathcal{F}_\varepsilon(\flot_N)$. Finally, using again equation \eqref{eq:kantorovich} we deduce that if $X_N^\varepsilon$ is any minimizer  $\rho_N^\varepsilon= (X_N^\varepsilon)_\# \rho_0$.
\end{proof}

Using the optimality conditions of the minimization problem \eqref{eq:mye}, one can actually provide an explicit expression for the minimizer $\rho_N^\varepsilon$ corresponding to an empirical measure $\rho_N$.  Such a characterization is proven in Proposition 11 in \cite{sarrazin2020lagrangian}, but we recall the precise statement in Proposition \ref{prop:gradient} below. In particular, this shows that $\rho_N^\varepsilon$ has a continuous bounded density on $M$. In turn, this allows us to prove the following statement which is a slight adaptation of Lemma 6.1 in \cite{evans1997partial}.

\begin{lemma}\label{lem:gradientF}
Let $X_N\in \mathbb{X}_N$ and define $X_N^\varepsilon$ and $\rho_N^\varepsilon$ as in Lemma \ref{lem:eulerianmy}.
For any $v \in C^{1}(M, \mathbb{R}^d)$ with $v\cdot n_{\partial M} = 0$ on $\partial M$, we have
\begin{equation}\label{eq:gradientf}
\int_{S_0} \frac{\flot_N-\flot_N^\varepsilon}{\varepsilon} \cdot v \circ \flot_N^\varepsilon \, \rho_0 \ed x = - \int_{M}  P(\rho^\varepsilon_N) \mathrm{div}\,v\, \ed x \,.
\end{equation}
\end{lemma}

\begin{proof}
We follow the proof of Lemma 6.1 in \cite{evans1997partial} and introduce first the flow of $v$, i.e.\ for $\delta> 0$ we define $Y:(-\delta,\delta)\times M\rightarrow M$ as the solution to the flow equation $\dot{Y}_s = v \circ Y_s $ for $s\in(-\delta,\delta)$ and with $Y_0 =\mathrm{Id}$, the identity map on $M$. Note that $Y_s:M\rightarrow M$ is a $C^{1}$ diffeomorphism, since $v$ is $C^{1}$ and it is tangent to the boundary, and we have
\begin{equation}\label{eq:jacobian}
{\partial_s}\, \mathrm{det}\nabla Y_s = (\mathrm{div}\, v \circ Y_s )\, \mathrm{det}\nabla Y_s\,.
\end{equation}
Then we define $\rho_s \coloneqq (Y_s)_\# \rho_N^\varepsilon$, and identifying $\rho_N^\varepsilon$ with its density with respect to $\ed x \mres M$ we have
\begin{equation}\label{eq:explicitpush}
\rho_s = \frac{\rho_N^\varepsilon}{\mathrm{det}\nabla Y_s} \circ Y_s^{-1}\,,
\end{equation}
which can be directly deduced  via a change variables in the integral formulation of the definition of the pushforward \eqref{eq:density}.
Moreover, the function
\[
g:s\in(-\delta,\delta) \rightarrow \frac{W^2_2(\rho_N,\rho_s)}{2\varepsilon} + \mathcal{U}(\rho_s) \in \mathbb{R}
\]
has a minimum at $s=0$. Since $\rho_N^\varepsilon$ is bounded, using equation \eqref{eq:explicitpush}, \eqref{eq:jacobian}, and the definition of $P$ in \eqref{eq:thermodynamic} we obtain 
\begin{equation}\label{eq:explicitderu}
\frac{\ed}{\ed s} \bigg |_{s=0} \mathcal{U}(\rho_s) = \frac{\ed}{\ed s} \bigg |_{s=0} \int_M U \left(\frac{\rho_N^\varepsilon}{\mathrm{det}\nabla Y_s}\right) \mathrm{det}\nabla Y_s \,\ed x = - \int_M P(\rho_N^\varepsilon) \mathrm{div}\,v \, \ed x\,.
\end{equation}
We now introduce $\gamma_s = (\nabla \psi ,Y_s)_\#\rho_N^\varepsilon$, so that $W^2_2(\rho_N,\rho_s) \leq \int |x-y|^2 \ed \gamma_s(x,y)$, which implies 
\[
\begin{aligned}
W^2_2(\rho_N,\rho_s) - W^2_2(\rho_N,\rho_N^\varepsilon) &\leq \int_{M} |\nabla \psi - Y_s|^2 \ed \rho_N^\varepsilon - \int_{M} |\nabla \psi - \mathrm{Id}|^2 \ed \rho_N^\varepsilon\\ & =\int_{M} ( Y_s - \mathrm{Id}) \cdot(\mathrm{Id} +Y_s - 2\nabla \psi) \ed \rho_N^\varepsilon\,.
\end{aligned}
\]
Therefore,
\[
0 \leq g(s) - g(0) \leq \frac{1}{2\varepsilon} \int_{M} ( Y_s - \mathrm{Id}) \cdot(\mathrm{Id} +Y_s - 2\nabla \psi) \ed \rho_N^\varepsilon +\mathcal{U}(\rho_s) - \mathcal{U}(\rho_N^\varepsilon)\,.
\]
Dividing by $s$, taking the limit for $s\rightarrow 0$ and using equation \eqref{eq:explicitderu} gives
\begin{equation}
\label{eq:ineqpres}
\int_{M} \frac{\nabla\psi-\mathrm{Id}}{\varepsilon} \cdot v  \, \ed \rho_N^\varepsilon \leq - \int_{M}  P(\rho^\varepsilon_N) \mathrm{div}\,v\, \ed x \,.
\end{equation}
Since the same also holds replacing $v$ by $-v$,  equality holds  and we obtain equation \eqref{eq:gradientf} by a change of variables on the left-hand side of \eqref{eq:ineqpres}.

\end{proof}

\begin{remark}\label{Re:caclulgradientformel}
Note that using the same computation of equation \eqref{eq:explicitderu}, and performing a change of variables on its right-hand side, one can formally identify $\nabla_{\mathbb{X}} \mathcal{F}(X_t)=  \nabla U'( \rho_t)\circ {\flot}_t$ in equation \eqref{eq:secondorder} and \eqref{eq:gradient}.
%
%

\end{remark}


\section{Modulated energy}\label{sec:modulated}

In this section we introduce the two main quantities that we will need to measure the distance between continuous and discrete solutions of problems \eqref{eq:euler} and \eqref{eq:porous}. These are constructed as discrete versions of the classical relative kinetic and internal energy of the system expressed in Eulerian variables. Here we adapt these definitions to our discrete Lagrangian setting and to the regularized energy \eqref{eq:my}.

The relative kinetic energy in the discrete setting is defined as follows:

\begin{definition}[Relative kinetic energy]\label{def:relkin}
Given a curve ${u}: [0,T] \rightarrow  C^0(\mathbb{R}^d; \mathbb{R}^d)$, the relative kinetic energy of a discrete flow $\flot_N: [0,T]\rightarrow \mathbb{X}_N$ with respect to $\barr{u}$ at time $t$ is given by
\begin{equation}
\begin{aligned}\label{eq:relkin}
\mathcal{K}_{\barr{u}}(t,\flot_N) & \coloneqq  \frac{1}{2} \|\dot{\flot}_N(t,\cdot) - \barr{u}(t,\flot_N(t,\cdot))\|^2_{\mathbb{X}} \\&=   \frac{1}{2} \sum_{i=1}^N |\dot{\flot}_N^i(t) - \barr{u}(t,\flot_N^i(t))|^2 \rho_0[P_i]\,.
\end{aligned}
\end{equation}
\end{definition}

\begin{remark} The choice of the relative kinetic energy in definition \ref{def:relkin} can be motivated as follows. The kinetic energy can be viewed as a convex function of the density $\rho$ and the momentum $m = \rho u$ given by
\begin{equation}\label{eq:kinmrho}
 \int_M \frac{|m|^2}{2\rho}\,.
\end{equation}
Then, it is natural to measure the distance between two states $(\rho,m)$ and $(\tilde{\rho},\tilde{m})$, with $\tilde{m} = \tilde{\rho} \tilde{u}$, by considering the difference between the value of the functional \eqref{eq:kinmrho} at $(\rho,m)$ and the linear part of its Taylor expansion at $(\tilde{\rho},\tilde{m})$ in the direction $(\rho-\tilde{\rho},m-\tilde{m})$. The resulting quantity is given by
\begin{equation}\label{eq:relkinsmooth}
\int_M \frac{1}{2}|u-\tilde{u}|^2 \rho \, \ed x\,,
\end{equation}
which is precisely the Eulerian counterpart to equation \eqref{eq:relkin}.
\end{remark}
In the order to define the relative internal energy in the discrete setting, for any $\rho, \tilde{\rho}\in C^0(M,(0,\infty))$ we first define
\begin{equation}\label{eq:relpot}
\mathcal{U}(\rho|\tilde{\rho}) \coloneqq \int_M U(\rho|\tilde{\rho}) \, \ed x \,, 
\end{equation}
where
\begin{equation}
U(r|s) \coloneqq U(r) - U(s)  - U'(s) (r-s)\,.
\end{equation}
If $|U'_+(0)|<+\infty$, equation \eqref{eq:relpot} defines $\mathcal{U}(\rho|\tilde{\rho})$ for any $\rho, \tilde{\rho}\in C^0(M,[0,\infty))$. Since we assume $U$ to be strictly convex, $\mathcal{U}(\rho|\tilde{\rho})\geq 0$ and it vanishes if and only if $\rho =\tilde{\rho}$.

The relative internal energy in the discrete setting is defined in order to fit the solutions of the numerical schemes detailed in Section \ref{sec:timediscrete}, and in particular the corresponding time discretization, which we recall in the definition below.

\begin{definition}[Discrete relative internal energy]
Let $\tau>0$ a fixed time step, $N_T \in \mathbb{N}^*$ be the number of time steps with $T = \tau N_T$, and $t_n \coloneqq n \tau$ for any $0\leq n \leq N_T$.  
Given a curve $\barr{\rho}: [0,T] \rightarrow  C^0(\mathbb{R}^d;[0,\infty))$, the discrete relative internal energy of a discrete flow $\flot_N: [0,T]\rightarrow \mathbb{X}_N$ with respect to $\barr{\rho}$ at time $t\in [t_n,t_{n+1})$ is given by
\begin{equation}\label{eq:modpotential}
\mathcal{F}_{\varepsilon,\barr{\rho}}(t,\flot_N) \coloneqq \frac{\|\flot_N(t)-\flot_N^\varepsilon(t_n)\|^2_{\mathbb{X}}}{2\varepsilon} + \mathcal{U}({\rho}^\varepsilon_N(t_n)| \barr{\rho}(t))\,,
\end{equation}
where ${\flot}^\varepsilon_N(t_n)$ is any fixed minimizer of problem \eqref{eq:my}, i.e.
\[
{\flot}^\varepsilon_N(t_n) \in \underset{\sigma \in \mathbb{X}}{\mathrm{argmin}} \frac{\|\flot_N(t_n)-\sigma \|^2_{\mathbb{X}}}{2\varepsilon}  + \mathcal{F}(\sigma)\,,\]
and ${\rho}^\varepsilon_N(t_n) \coloneqq ({\flot}^\varepsilon_N(t_n))_\# \rho_0$.
 \end{definition}

\begin{remark} In the smooth Eulerian setting the relative internal energy would be given just by the functional in equation \eqref{eq:relpot}. Importantly, even if the potential energy of the discrete system is a convex functional on $\mathbb{X}$, the discrete relative internal energy in \eqref{eq:modpotential} does not correspond to this point of view and should rather be regarded as an approximation of \eqref{eq:relpot}. The same holds for the definition of the relative kinetic energy above, which does not coincide with the one obtained interpreting the kinetic energy as convex functional on $\mathbb{X}$. This time however there is no approximation since if we replaced $X_N$ by a smooth injective flow we could recover \eqref{eq:relkinsmooth} from \eqref{eq:relkin} by a simple change of variables.
\end{remark}

The convergence proof in Section \ref{sec:convergence} will rely on a Gr\"onwall argument based on the discrete relative energies \eqref{eq:relkin} and \eqref{eq:modpotential}. It will require us to control the time variation of the total discrete relative energy by itself. The advantage of adopting an Eulerian rather than Lagrangian point of view in the definitions above is that, in the Eulerian case, such a control can be enforced by exploiting simple algebraic properties of the functions $P$ and $U$.  More precisely, we will need to control the relative pressure
\begin{equation}\label{eq:relpress}
P(r|s) \coloneqq P(r) - P(s)  - P'(s) (r-s)\,
\end{equation}
by $U(r|s)$. To this end, we will make the following assumption: there exists a constant $A>0$ such that
\begin{equation}\label{eq:pbound}
|P''(r)|  \leq A\, U''(r)\, \quad \forall\, r> 0\,.
\end{equation}
This assumption is verified in the classical cases of interest of power laws and of the entropy. It implies the following lemma, which is an extract of Lemma 3.3 in \cite{giesselmann2017relative}.

\begin{lemma} \label{lem:Abound}
Let $U$ and $P$ be smooth functions on $[0,\infty)$ verifying \eqref{eq:thermodynamic} and \eqref{eq:pbound}. Then
\begin{equation}\label{eq:relpbound}
|P(r|s)| \leq A U(r|s) \, \quad \forall\, r,s>0\,.
\end{equation}
\end{lemma}
\begin{proof}
We have $P(r|s) = (r-s)^2 \int_0^1 (1-\theta) P''((1-\theta)s + \theta r) \,\ed \theta$ and similarly for $U(r|s)$. Hence, using equation \eqref{eq:pbound},
\[
|P(r|s)|\leq (r-s)^2 \int_0^1 (1-\theta) |P''((1-\theta)s + \theta r)| \,\ed \theta \leq A\, U(r|s)\,.
\]
\end{proof}

\begin{remark}\label{lem:positivity}
In the following, in order to treat the case of the convergence towards solutions with vanishing density we will need to add the hypothesis that $U$ admits a right third derivative at $0$, i.e. $|U'''_+(0)|<\infty$. Note that in this setting, if equation \eqref{eq:relpbound} holds for $r,s>0$, then it holds by continuity for $r,s\geq 0$. 
\end{remark}

\section{Convergence of the fully discrete scheme}\label{sec:convergence}

In this section we use the discrete relative energies introduced in Section \ref{sec:modulated} to prove our convergence results for the space-time discretization of problems \eqref{eq:euler} and \eqref{eq:porous} defined in Section \ref{sec:timediscrete}.

 Since the image of the discrete solution $\flot_N(t)$ (i.e.\ the particles' positions) may not be contained in the domain $M$, an essential ingredient of the proof is the possibility to extend the exact solution of the continuous models outside the domain. Importantly, besides keeping the same regularity, the extended density and veloctiy will need to satisfy the continuity equation also outside the domain. We construct such extended variables in the following lemma, by exploting the properties of the continuity equation and using an extension theorem due to Fefferman \cite{fefferman2009extension}.

\begin{lemma}\label{lem:continuityextension}
Let $u : [0,T]\times M \rightarrow \mathbb{R}^d$ be such that  $u \cdot n_{\partial M} = 0$ on $[0,T]\times \partial M$, and $\rho_0 : M \rightarrow [0,\infty)$. If $u$ is of class $C^{2,1}$ in space, uniformly in time, and $\rho_0$ is of class $C^{1,1}$, then there exist $\tilde{u}: [0,T]\times \mathbb{R}^d \rightarrow \mathbb{R}^d$ and  $\tilde{\rho}: [0,T] \times \mathbb{R}^d \rightarrow \mathbb{R}$ such that:
\begin{enumerate}
\item $\tilde{u}$ is an extension of $u$, i.e.\ $\tilde{u}(t)|_M = u(t)$ for all $t\in [0,T]$, and there exists a constant $C>0$ only depending on $d$ such that
 \begin{equation} \label{eq:boundu1} 
 \sup_{t\in[0,T]}  \| \tilde{u}(t) \|_{C^{2,1}} \leq C \sup_{t\in[0,T]} \| {u}(t) \|_{C^{2,1}}\, ;
 \end{equation} moreover, if $u\in C^1([0,T],C^{2,1}(M,\mathbb{R}^d))$ then
 \begin{equation} \label{eq:boundu2}
  \sup_{t\in[0,T]}  \| \partial_t \tilde{u}(t) \|_{C^{2,1}} \leq C \sup_{t\in[0,T]} \| \partial_t {u}(t) \|_{C^{2,1}}\,;
  \end{equation}
\item the couple $(\tilde{\rho},\tilde{u})$ solves the continuity equation:
\[
\partial_t \tilde{\rho} + \mathrm{div} (\tilde{\rho} \tilde{u}) = 0 \quad \text{on }\, [0,T] \times \mathbb{R}^d,
\]
and in particular the curve $\rho:t\in[0,T]\rightarrow \tilde{\rho}(t)|_M$ is the unique solution to the continuity equation on $[0,T] \times M$ associated with $u$ and initial conditions $\rho(0) = \rho_0$;  if $\rho_0 \geq \rho_{min}>0$, then  $\tilde{\rho}\geq \tilde{\rho}_{min}>0$, where $\tilde{\rho}_{min}$ only depend on $\rho_{min}$, $\sup_{t\in[0,T]} \|u(t)\|_{C^{2,1}}$, $T$ and $d$; moreover, $\sup_{t\in[0,T]} \|\tilde{\rho}(t)\|_{C^{1,1}}$ only depends on $\|\rho_0\|_{C^{1,1}}$, $\sup_{t\in[0,T]} \|u(t)\|_{C^{2,1}}$, $T$, $d$ (and on $\rho_{min}$ in the case $\rho_0 \geq \rho_{min}>0$).
\end{enumerate}
\end{lemma}

\begin{proof}
The first part is just an application of the construction proposed by Fefferman in \cite{fefferman2009extension} to extend H\"older continuous functions. In particular, by theorem 2 in \cite{fefferman2009extension}, for any $k\geq 0$ there exists a linear bounded operator $L_k: C^{k,1}(M) \rightarrow C^{k,1}(\mathbb{R}^d)$ such that the norm of $L_{k}$ is bounded by a  constant depending only on $d$ and $k$, and for any $f\in C^{k,1}(\mathbb{R}^d)$ one has $L_k f|_M =f$. Then, setting $\tilde u(t) = L_{2} \, u(t)$ (applied component-wise) for all $t \in[0,T]$ for a given extension operator $L_2$, we obtain the estimate \eqref{eq:boundu1} by the boundedness of $L_2$. In the case where  $u\in C^1([0,T],C^{2,1}(M,\mathbb{R}^d))$, by linearity of $L_2$ we have $\partial_t\tilde{u} = L_2 \partial_t u$, from which we deduce \eqref{eq:boundu2}.

For the second part, we first introduce  ${X}:[0,T]\times \mathbb{R}^d \rightarrow \mathbb{R}^d$ the flow of $\tilde{u}$, i.e.\ the solution to the flow equation $\dot{X}_t = \tilde{u}(t, X_t)$ with initial conditions $X_0 = \mathrm{Id}$. 
For all times $t\in[0,T]$, $X_t$ is a $C^{2,1}$ diffeomorphism of $\mathbb{R}^d$ and by construction the $C^{2,1}$ norm of $X_t$ and $X_t^{-1}$ only depend on that of $u$ and on $T$. 
Note, in particular, that the Jacobian determinant solves
\[
\partial_t \det \nabla X_t = \mathrm{div}\, \tilde{u}(t, X_t) \, \det \nabla X_t\,,
\]
which implies that for all $(t,x)\in[0,T]\times \mathbb{R}^d$,
\begin{equation}\label{eq:jacobbound}
\max \left \{\mathrm{det}\nabla X_t(x), \frac{1}{\mathrm{det}\nabla X_t(x)}\right\} \leq \exp \left(\int_0^t \|\mathrm{div} \,\tilde{u}(t) \|_{\infty} \,\ed t\right)\,.
\end{equation}
Now, if $\rho_0$ is not strictly-positive, we define an extension $\tilde{\rho}_0:\mathbb{R}^d\rightarrow \mathbb{R}$ of $\rho_0$ on the whole space by $\tilde{\rho}_0\coloneqq L_{1} \rho_0$ (and note that $\tilde{\rho}_0$ may be negative). Then,
we define for all $t\in[0,T]$
\begin{equation}\label{eq:rhotilde}
\tilde{\rho}(t) = \frac{\tilde{\rho}_0}{\det \nabla X_t}\circ X_t^{-1}\,,
\end{equation}
and therefore the regularity of $\tilde{\rho}$ in space derives from that of $\tilde{\rho}_0$, $X_t^{-1}$ and $\det \nabla X_t$, and from the bound \eqref{eq:jacobbound}.
Moreover, by direct computation one can check that $\tilde{\rho}$ solves the continuity equation with velocity $\tilde{u}$. On the other hand, if $\rho_0 \geq \rho_{min} > 0$, we define $\tilde{\rho}_0\coloneqq \exp( L_{1} \log(\rho_0))$ and $\tilde{\rho}$ as above. Then, the lower bound on $\tilde{\rho}$ can be deduced from equations \eqref{eq:rhotilde} and \eqref{eq:jacobbound}.
\end{proof}


In the following, we finally prove Theorem \ref{th:convsemieuler} and \ref{th:convsemigradient}, which establish a bound on the rate of convergence for our space-time discretizations of problems \eqref{eq:euler} and \eqref{eq:porous}, respectively. 

\begin{proof}[Proof of Theorem \ref{th:convsemieuler}] Throughout the proof we will denote by $\langle \cdot, \cdot\rangle$  and $\|\cdot\|$ the inner product and norm on $\mathbb{X}$, respectively, i.e.\ the $L^2$ inner product and norm weighted by $\rho_0$. Moreover, for any function $f:[0,T] \rightarrow C^{0,1}(E)$ with $E\subseteq \mathbb{R}^d$, we will denote by $\mathrm{Lip}_T(f) \coloneqq \sup_{t\in[0,T]} \mathrm{Lip}( f(t))$ and we will use the same notation for vector-valued functions.

We denote by $\tilde{u}$ and $\tilde{\rho}$ the extensions of $u$ and $\rho$, respectively, constructed via Lemma \ref{lem:continuityextension}. Note that if $\rho$ is not strictly-positive, $\tilde{\rho}$ may be negative. However, since in the case we suppose that $|U'''_+(0)|<+\infty$, we replace $U$ by a $C^3$ extension defined on $\mathbb{R}$ (which we still denote by $U$ with an abuse of notation), e.g., by setting $U(r) = \sum_{n=0}^3 U^{(n)}_+(0)r^n/n!$ for $r<0$ . Then, $U^{(n)}(\tilde{\rho})$ is Lipschitz in space, uniformly in time, for $n=0,1,2$.

We define the relative energy as follows:
\begin{equation}\label{eq:modenergy}
\mathcal{E}_{\barr{\rho},\barr{u}}(t,\flot_N) \coloneqq \mathcal{K}_{\tilde{u}}(t,\flot_N) +\mathcal{F}_{\varepsilon,\barr{\rho}}(t,\flot_N) + \frac{1}{2}\| \flot_N(t) - \flot(t) \|^2\,.
\end{equation}
Note that besides the relative kinetic and internal energy, we also included an additional term in \eqref{eq:modenergy} given by the squared $L^2$ distance between the flows and which will help us deal with the fact that the image of $\flot_N(t)$ may not be included in $M$.  Note also that while the relative kinetic energy needs to be computed using the extended velocity field $\tilde{u}$, for the relative internal energy we can use indifferently either $\rho$ or $\tilde{\rho}$ since it is defined via an integral over the (fixed) domain $M$.

The strategy of the proof is the following. First of all, we compute separately the time derivative of the three terms in \eqref{eq:modenergy} for $t\in[t_n,t_{n+1})$. We then apply Gr\"onwall’s inequality on the same time interval to obtain a first estimate. Finally, we use a discrete Gr\"onwall's inequality for $0\leq n \leq N_T$ to prove the result. 

%

\paragraph{\textbf{Step 1: Time derivative of the relative kinetic energy}} We introduce the material derivative
\[
D_t \tilde{u}(t) \coloneqq \partial_t \tilde{u}(t) + \tilde{u}(t) \cdot \nabla \tilde{u}(t)\,.
\]
Then, using equation \eqref{eq:disceuler}, we have
\begin{equation}\label{eq:derkinetic0}
\begin{aligned}
\frac{\ed}{\ed t}  \mathcal{K}_{\tilde{u}}(t,\flot_N) & = \langle \ddot{\flot}_N(t) - \partial_t \tilde{u}(t,\flot_N(t)) - \dot{\flot}_N(t) \cdot \nabla \tilde{u}(t,\flot_N(t)), \dot{\flot}_N(t) - \tilde{u}(t,\flot_N(t))\rangle \\
& = -\langle (\dot{\flot}_N(t) -  \tilde{u}(t,\flot_N(t))) \cdot \nabla \tilde{u}(t,\flot_N(t)), \dot{\flot}_N(t) - \tilde{u}(t,\flot_N(t))\rangle \\
& \quad  - \langle  \varepsilon^{-1} ( \flot_N(t)-\flot_N^\varepsilon(t_n)) + D_t\tilde{u}(t, \flot_N(t))), \dot{\flot}_N(t) - \tilde{u}(t,\flot_N(t))\rangle\,,
\end{aligned}
\end{equation}
where we replaced $ \ddot{\flot}_N(t) $ using \eqref{eq:disceuler}, and we removed the projection onto $\mathbb{X}_N$, since  $\dot{\flot}_N(t) - \tilde{u}(t,\flot_N(t)) \in \mathbb{X}_N$. Observe that the system \eqref{eq:euler} implies 
\[
\rho_0 D_t \tilde{u}(t,\flot(t)) = - \rho_0 \nabla U'(\tilde{\rho}(t,\flot(t))\,.
\]
Then, adding and subtracting $ \nabla U'(\tilde{\rho}(t, \flot(t)))$ and $ \nabla U'(\tilde{\rho}(t, \flot_N(t)))$ in the last inner product in \eqref{eq:derkinetic0}, we obtain
\begin{equation}\label{eq:derkinetic}
\begin{aligned}
\frac{\ed}{\ed t}  \mathcal{K}_{\tilde{u}}(t,\flot_N)  =& -\langle (\dot{\flot}_N(t) -  \tilde{u}(t,\flot_N(t))) \cdot \nabla \tilde{u}(t,\flot_N(t)), \dot{\flot}_N(t) - \tilde{u}(t,\flot_N(t))\rangle \\
& - \langle D_t\tilde{u}(t, \flot_N(t)))-D_t\tilde{u}(t, \flot(t))), \dot{\flot}_N(t) - \tilde{u}(t,\flot_N(t))\rangle\\
& + \langle \nabla U'(\tilde{\rho}(t, \flot(t)))-\nabla U'(\tilde{\rho}(t, \flot_N(t))), \dot{\flot}_N(t) - \tilde{u}(t,\flot_N(t))\rangle\\
& - \langle\varepsilon^{-1} ( \flot_N(t)-\flot_N^\varepsilon(t_n))-\nabla U'(\tilde{\rho}(t, \flot_N(t))), \dot{\flot}_N(t) - \tilde{u}(t,\flot_N(t))\rangle\,.
\end{aligned}
\end{equation}


\paragraph{\textbf{Step 2: Time derivative of the relative internal energy}} 
First of all, we define the following quantity which will be useful for the computations below and also later in the Gr\"onwall argument (see also Remark \ref{rem:relentropy} below):
\begin{equation}\label{eq:Hndef}
H^n(t) \coloneqq \int_{\mathbb{R}^d} U'(\tilde{\rho}(t)) \ed ( \rho_N^\varepsilon(t_n ) -  \rho_N(t))\,.
\end{equation}

We now compute the time derivatives of the different terms in $\mathcal{F}_{\varepsilon,\rho} (t,X_N)$ (defined by equation \eqref{eq:modpotential}) for $t\in[t_n,t_{n+1})$.
By the same computations as in \eqref{eq:explicitderu}, we have
\begin{equation}\label{eq:derfirstorder0}
\frac{\ed}{\ed t} \mathcal{U}(\barr{\rho}(t)) = -\int_M P(\barr{\rho}(t)) \mathrm{div}\,\barr{u}(t) \, \ed x \,.
\end{equation}
For the time derivative of the discrete energy we can arrange the terms in order to obtain a similar quantity. In particular, we have 
\begin{equation}\label{eq:derfirstorder}
\begin{aligned}
\MoveEqLeft[6] \frac{\ed}{\ed t} \left( \frac{\|\flot_N(t) - \flot^\varepsilon_N(t_n)\|^2}{2\varepsilon} + \mathcal{U}(\rho_N^\varepsilon(t_n) )\right) & \\ & = \varepsilon^{-1} \langle \flot_N(t)-\flot_N^\varepsilon(t_n), \dot{\flot}_N(t)\rangle \\
& = \varepsilon^{-1}\langle \flot_N(t)-\flot_N^\varepsilon(t_n), \dot{\flot}_N(t) - \tilde{u}(t,\flot_N(t)) \rangle\\
& \quad +\varepsilon^{-1} \langle \flot_N(t)-\flot_N^\varepsilon(t_n), \tilde{u}(t,\flot_N(t)) - \tilde{u}(t, \flot_N^\varepsilon(t_n))\rangle\\
& \quad + \varepsilon^{-1} \langle \flot_N(t) - \flot_N(t_n),  {u}(t, \flot_N^\varepsilon(t_n))\rangle \\&\quad + \varepsilon^{-1}\langle \flot_N(t_n) - \flot_N^\varepsilon(t_n),  {u}(t, \flot_N^\varepsilon(t_n))\rangle\,,
\end{aligned}
\end{equation}
and note that by Lemma \ref{lem:gradientF}, the last term in in equation \eqref{eq:derfirstorder} can be written as follows
\begin{equation}\label{eq:optimalitypres}
\varepsilon^{-1}\langle \flot_N(t_n) - \flot_N^\varepsilon(t_n),  {u}(t, \flot_N^\varepsilon(t_n))\rangle= -\int_M P(\rho_N^\varepsilon(t_n)) \mathrm{div}\,\barr{u}(t) \, \ed x \,.
\end{equation}

We write the time derivative of the remaining term in $\mathcal{F}_{\varepsilon,\rho}(t, X_N)$ as follows:
\begin{equation}\label{eq:derusplit}
\frac{\ed}{\ed t}   \int_M  U'({\rho}(t))( \rho^\epsilon_N(t_n) - \barr{\rho}(t)) \, \ed x =  \frac{\ed}{\ed t} H^n(t) +  \frac{\ed}{\ed t}   \int_{\mathbb{R}^d}  U'(\tilde{\rho}(t))\ed( \rho_N(t) - \barr{\rho}(t)).
\end{equation}
Note that here we identify $\rho(t)$ with a measure on $\mathbb{R}^d$ extending it by zero, and we will use the same convention also in the following. Then, we compute
\begin{equation}\label{eq:dersecondorder}
\begin{aligned} 
\frac{\ed}{\ed t}   \int_{\mathbb{R}^d} U'(\tilde{\rho}(t))\, \ed ( \rho_N(t) - \barr{\rho}(t))  =&   \langle \nabla U'(\tilde{\rho}(t)) \circ \flot_N(t), \dot{\flot}_N(t)\rangle\\ &- \int_M \barr{u}(t) \cdot \nabla U'(\barr{\rho}(t)) \barr{\rho}(t)\,\ed x\\ & -\int_{\mathbb{R}^d}  U''(\tilde{\rho}(t))\mathrm{div}(\tilde{\rho}(t)\tilde{u}(t)) \,\ed (\rho_N(t) -\rho(t)) .
\end{aligned}
\end{equation}
Remark that here we used the fact that the continuity equation holds also for the extended functions $(\tilde{\rho},\tilde{u})$, which is due to the construction described in Lemma \ref{lem:continuityextension}. 
Using  $\mathrm{div}(\tilde{\rho}\tilde{u}) = \tilde{\rho} \,\mathrm{div} \tilde{u} + \nabla \tilde{\rho} \cdot \tilde{u}$ and then using $P'(r) = r U''(r)$, we get 
\begin{equation}\label{eq:dersecondordertemp}
\begin{aligned}
\frac{\ed}{\ed t}   \int_{\mathbb{R}^d} U'(\tilde{\rho}(t))\, \ed ( \rho_N(t) - \barr{\rho}(t)) & =
 \langle \nabla U'(\tilde{\rho}(t)) \circ \flot_N(t), \dot{\flot}_N(t)-\tilde{u}(t,\flot_N(t))\rangle\\ 
&\quad -\int_{\mathbb{R}^d}  P'(\tilde{\rho}(t))\mathrm{div}\tilde{u}(t) \ed ( \rho_N(t)  - \barr{\rho}(t)) \,.
\end{aligned}
\end{equation}
Putting this back into equation \eqref{eq:derusplit}, we find
\begin{equation}\label{eq:dersecondorder1}
\begin{aligned}
\frac{\ed}{\ed t}   \int_M  U'(\barr{\rho}(t))( \rho^\epsilon_N(t_n) - \barr{\rho}(t)) \, \ed x & =  \frac{\ed}{\ed t} H^n(t) + 
 \langle \nabla U'(\tilde{\rho}(t)) \circ \flot_N(t), \dot{\flot}_N(t)-\tilde{u}(t,\flot_N(t))\rangle\\ 
&\quad -\int_{\mathbb{R}^d}  P'(\tilde{\rho}(t))\mathrm{div}\tilde{u}(t) \ed ( \rho_N(t) -\rho^\epsilon_N(t_n) ) \\
& \quad -\int_M  P'(\barr{\rho}(t)) \mathrm{div}\,\barr{u}(t) ( \rho^\epsilon_N(t_n) - \barr{\rho}(t)) \, \ed x \,.
\end{aligned}
\end{equation}
Note that we have added and subtracted  $\rho^\epsilon_N(t_n)$ in the last integral, which allows us to retrieve $P(\rho_N^\varepsilon(t_n)|\rho)$ when combining all terms. In fact, replacing  \eqref{eq:optimalitypres} into  \eqref{eq:derfirstorder}, and subtracting the contributions from \eqref{eq:derfirstorder0} and \eqref{eq:dersecondorder1}, we obtain
\begin{equation}\label{eq:derpotential}
\begin{aligned}
\frac{\ed}{\ed t} \mathcal{F}_{\varepsilon,\barr{\rho}}(t,\flot_N) & =  \langle  \varepsilon^{-1} (\flot_N(t)-\flot_N^\varepsilon(t_n)) -\nabla U'(\tilde{\rho}(t, \flot_N(t))), \dot{\flot}_N(t) - \tilde{u}(t,\flot_N(t)) \rangle\\
& \quad +  \varepsilon^{-1}\langle  \flot_N(t)-\flot_N^\varepsilon(t_n), \tilde{u}(t,\flot_N(t)) - \tilde{u}(t,\flot_N^\varepsilon(t_n))\rangle\\
& \quad - \int_M P(\rho^\varepsilon_N(t_n)|\barr{\rho}(t)) \mathrm{div}\,\barr{u}(t) \, \ed x\\& \quad 
+ \int_{\mathbb{R}^d} P'(\tilde{\rho}(t))\mathrm{div}(\tilde{u}(t)) \ed ( \rho_N(t) -\rho^\epsilon_N(t_n) )  \\
& \quad +\varepsilon^{-1} \langle \flot_N(t) - \flot_N(t_n),  \tilde{u}(t,\flot_N^\varepsilon(t_n))\rangle -\frac{\ed}{\ed t} H^n(t) \,.
\end{aligned}
\end{equation}
We finally observe that the first term on the right-hand side of equation \eqref{eq:derpotential} coincides with the opposite of the last term in \eqref{eq:derkinetic}. Therefore the two terms cancel out when adding the two equations. The decomposition of the time derivative in \eqref{eq:derfirstorder} is designed to exploit this feature, which is a consequence of energy conservation.
  

\paragraph{\textbf{Step 3: Gr\"onwall's argument on $[t_n,t_{n+1})$}} Combining
\[
\frac{\ed}{\ed t} \frac{1}{2} \| \flot_N(t) - \flot(t)\|^2  = \langle \dot{\flot}_N(t) - \dot{\flot}(t), \flot_N(t) - \flot(t)\rangle 
\]
with equations \eqref{eq:derkinetic} and \eqref{eq:derpotential}, we obtain
\begin{equation}\label{eq:dertotal}
\begin{aligned}
\frac{\ed}{\ed t} \mathcal{E}_{\rho,\barr{u}}(t,\flot_N) & =  - \langle D_t\tilde{u}(t, \flot_N(t)))-D_t\tilde{u}(t, \flot(t))), \dot{\flot}_N(t) - \tilde{u}(t,\flot_N(t))\rangle\\
&\quad + \langle \nabla U'(\tilde{\rho}(t, \flot(t)))-\nabla U'(\tilde{\rho}(t, \flot_N(t))), \dot{\flot}_N(t) - \tilde{u}(t,\flot_N(t))\rangle\\
& \quad -\langle (\dot{\flot}_N(t) -  \tilde{u}(t,\flot_N(t))) \cdot \nabla \tilde{u}(t,\flot_N(t)), \dot{\flot}_N(t) - \tilde{u}(t,\flot_N(t))\rangle\\
& \quad + \langle \dot{\flot}_N(t) - \dot{\flot}(t), \flot_N(t) - \flot(t)\rangle \\
& \quad +  \varepsilon^{-1}\langle  \flot_N(t)-\flot_N^\varepsilon(t_n), \tilde{u}(t,\flot_N(t)) - \tilde{u}(t,\flot_N^\varepsilon(t_n))\rangle \\ 
& \quad - \int_M P(\rho^\varepsilon_N(t_n)|\barr{\rho}(t)) \mathrm{div}\,\barr{u}(t) \, \ed x \\
& \quad +\int_{\mathbb{R}^d} P'(\tilde{\rho}(t))\mathrm{div}(\tilde{u}(t)) \ed ( \rho_N(t) -\rho^\epsilon_N(t_n) ) \\
& \quad + \varepsilon^{-1} \langle \flot_N(t) - \flot_N(t_n),  {u}(t,\flot_N^\varepsilon(t_n))\rangle -  \frac{\ed}{\ed t} H^n(t)\\
& \eqqcolon J_1 + J_2+ J_3 +J_4 + J_5 +J_6+J_7 + J_8  - \frac{\ed}{\ed t} H^n(t)\,.
\end{aligned}
\end{equation}

Applying  Cauchy-Schwarz and then Young's inequality to the first two terms we obtain
\begin{equation}\label{eq:J12}
J_1 + J_2 \leq 2 (\mathrm{Lip}_T(D_t \tilde{u})  + \mathrm{Lip}_T( \nabla U'(\tilde{\rho}))) \left( \mathcal{K}_{\tilde{u}}(t,\flot_N) + \frac{1}{2}\| \flot_N(t) - \flot(t) \|^2\right)\,,
\end{equation}
where $D_t \tilde{u}$ and $\nabla U'(\tilde{\rho})$ are interpreted as functions on $[0,T] \times \mathbb{R}^d$.

For $J_4$, we have
\begin{equation}\label{eq:derflows}
\begin{aligned}
J_4
& = \langle \dot{\flot}_N(t) - \tilde{u}(t,\flot_N(t)), \flot_N(t) - \flot(t)\rangle
\\ &\quad +  \langle \tilde{u}(t,\flot_N(t)) - \tilde{u}(t,\flot(t)), \flot_N(t) - \flot(t)\rangle\\
 & \leq \mathcal{K}_{\tilde{u}}(t,\flot_N) + \left(1+ 2 \mathrm{Lip}_T(\tilde{u})\right) \frac{1}{2} \| \flot_N(t) - \flot(t)\|^2 \,,
\end{aligned}
\end{equation}

Using the estimate \eqref{eq:derflows}, we find 
\begin{equation}\label{eq:J3456}
\begin{aligned}
\sum_{i=3}^6 J_i & \leq (1+ 2 \mathrm{Lip}_T(\tilde{u}))\left ( \mathcal{K}_{\tilde{u}}(t,\flot_N) + \frac{\|\flot_N(t)-\flot(t)\|^2}{2} \right) \\
& \quad + \mathrm{Lip}_T(\tilde{u}) \frac{\|\flot_N(t)-\flot_N^\varepsilon(t_n)\|^2}{\varepsilon} + A \mathrm{Lip}_T(\tilde{u}) \mathcal{U}(\rho_N^\varepsilon(t_n)| \barr{\rho}(t)) \\
& \leq (1+ A'\mathrm{Lip}_T(\tilde{u})) \mathcal{E}_{\rho,u}(t,\flot_N) \,,
\end{aligned} 
\end{equation}
where $A'\coloneqq \max(A,2)$, and we used for $J_6$ the inequality given in Lemma \ref{lem:Abound} (see also Remark \ref{lem:positivity}).
Hence, combining \eqref{eq:J12} and \eqref{eq:J3456} we obtain
\begin{equation}\label{eq:J16}
\sum_{i=1}^6 J_i \leq C_1 \mathcal{E}_{\barr{\rho},\barr{u}}(t,\flot_N)\,,
\end{equation}
where $C_1 \coloneqq 2 \mathrm{Lip}_T(D_t \tilde{u}) + 2\mathrm{Lip}_T(\nabla U'(\tilde{\rho})) + A' \, \mathrm{Lip}_T(\tilde{u}) +1 $.
For $J_7$ we have
\begin{equation}\label{eq:J7}
\begin{aligned}
J_7 & \leq \mathrm{Lip}_T(P'(\tilde{\rho})\, \mathrm{div}\,\tilde{u}) W_1(\rho_N(t),\rho_N^\varepsilon(t_n)) \\ & \leq C_2 \left(\frac{\varepsilon}{2} + \frac{W_2^2(\rho_N(t),\rho_N^\varepsilon(t_n))}{2\varepsilon}\right) \\ 
& \leq C_2 \left(\frac{\varepsilon}{2} + \frac{\| \flot_N(t) - \flot_N^\varepsilon(t_n)\|^2}{2\varepsilon}\right) \,,
\end{aligned}
\end{equation}
where $W_1(\cdot,\cdot)$ denotes the $L^1$-Wasserstein distance and we have used the inequality $W_1(\rho_N(t),\rho_N^\varepsilon(t_n)) \leq W_2(\rho_N(t),\rho_N^\varepsilon(t_n))$ (see Chapter 5 in \cite{santambrogio2015optimal}), and where $C_2\coloneqq \mathrm{Lip}_T(P'(\tilde{\rho})\, \mathrm{div}\,\tilde{u})$. 

%

For $J_8$ we have
\begin{equation}\label{eq:J8bound}
\begin{aligned}
J_8 & = \frac{1}{\varepsilon} \int_{t_n}^t \langle \dot{\flot}_N(t'), u(t,\flot_N^\varepsilon(t_n)\rangle \, \ed t' \\
    & \leq  \frac{1}{\varepsilon} \int_{t_n}^t \| \dot{\flot}_N(t') \| \| u(t,\flot_N^\varepsilon(t_n) )\|\, \ed t' \\
    & \leq \frac{\tau}{\varepsilon} \left( \mathcal{E}_\varepsilon(t_n,\flot_N)  - \mathrm{min}\, \mathcal{U} +\frac{1}{2}\| u\|^2_{L^\infty([0,T]\times M)}\right) \\
    & \leq \frac{\tau}{\varepsilon} \left( \mathcal{E}_\varepsilon(0,\flot_N)  +C_3 \right) \,,
\end{aligned}
\end{equation}
where we used the conservation-dissipation of the energy $\mathcal{E}_\varepsilon$ \eqref{eq:energystab} for the last two inequalities,  and where $C_3\coloneqq \| u\|^2_{L^\infty([0,T]\times M)}/2- \mathrm{min}\, \mathcal{U}$. 

Using the same argument as for $J_7$, we obtain
\begin{equation}\label{eq:Hn}
\begin{aligned}
|H^n(t)| \leq & {|\mathrm{Lip}_T(U'(\tilde{\rho}))|^2}\varepsilon+ \frac{W_2^2(\rho_N(t),\rho_N^\varepsilon(t_n))}{4\varepsilon} \\
\leq & C_4 \varepsilon + \frac{1}{2} \mathcal{E}_{\rho,u}(t,\flot_N)\,,
\end{aligned}
\end{equation}
where $C_4\coloneqq  {|\mathrm{Lip}_T(U'(\tilde{\rho}))|^2}$.

This last inequality allows us to include $H^n(t)$ in the Gr\"onwall argument. In particular, let $E^n(t)\coloneqq \mathcal{E}_{\rho,u}(t,\flot_N) + H^n(t)$. Combining the  estimates \eqref{eq:J16}, \eqref{eq:J7}, \eqref{eq:J8bound}, into \eqref{eq:dertotal}, we find
\[
\frac{\ed}{\ed t} E^n(t) \leq (C_1+C_2 ) \mathcal{E}_{\rho,u}(t,\flot_N) + \frac{C_2}{2} \varepsilon + (C_3 + \mathcal{E}_\varepsilon(0,\flot_N) ) \frac{\tau}{\varepsilon}\,.
\]
Adding and subtracting $ 2 (C_1+C_2 ) H^n(t)$, using the bound \eqref{eq:Hn} and rearranging terms, this implies
\[
\begin{aligned}
\frac{\ed}{\ed t} E^n(t) \leq & 2 (C_1+C_2 ) E^n(t) + (\frac{C_2}{2} + 2 (C_1+C_2 )C_4) \varepsilon +  (C_3 + \mathcal{E}_\varepsilon(0,\flot_N) ) \frac{\tau}{\varepsilon} \\ 
\eqqcolon & C_5 E^n(t) + C_6 \varepsilon +(C_3 + \mathcal{E}_\varepsilon(0,\flot_N) )  \frac{\tau}{\varepsilon}  \,.
\end{aligned}
\]
Applying Gr\"onwall inequality over the interval $[t_n,s]$ with $t_n<s <t_{n+1}$, we obtain
\[
E^n(t_{n+1}^-) \coloneqq \lim_{s\rightarrow t_{n+1}^-}  E^n(s) \leq  (E^n(t_n)+ C_6  \varepsilon \tau +  (C_3 + \mathcal{E}_\varepsilon(0,\flot_N) ) \frac{\tau^2}{\varepsilon}  ) \exp(C_5\tau)\,.
\]
In order to apply a discrete Gr\"onwall inequality, we need to replace the left-hand side with $E^{n+1}(t_{n+1})\coloneqq \mathcal{E}_{\rho,u}(t_{n+1},\flot_N) + H^{n+1}(t_{n+1})$.  This is indeed possible, since by definition of $\flot_N^\varepsilon(t_{n+1})$ and continuity of $\rho, {\rho}_N$ and $\flot_N^\varepsilon $ we have 
\begin{equation}\label{eq:fplush}
\begin{aligned}
\mathcal{F}_{\varepsilon,\barr{\rho}}(t_{n+1},\flot_N) + H^{n+1}(t_{n+1}) & =
 \frac{ \| \flot_N(t_{n+1}) - \flot_N^\varepsilon(t_{n+1})\|^2 }{2\varepsilon} +\mathcal{U}({\rho}^\varepsilon_N(t_{n+1}) )  - \mathcal{U}(\barr{\rho}(t_{n+1})) \\ & \quad- \int_{\mathbb{R}^d}  U'(\tilde{\rho}(t_{n+1})) \ed ({\rho}_N(t_{n+1})-\textbf{1}_M \barr{\rho}(t_{n+1})) \\
 & \leq \frac{\| \flot_N(t_{n+1}) - \flot_N^\varepsilon(t_{n})\|^2 }{2\varepsilon} +\mathcal{U}({\rho}^\varepsilon_N(t_{n}) )  - \mathcal{U}(\barr{\rho}(t_{n+1})) \\ & \quad- \int_{\mathbb{R}^d}  U'(\tilde{\rho}(t_{n+1})) \ed ({\rho}_N(t_{n+1})- \textbf{1}_M \barr{\rho}(t_{n+1})) \\
 & = \mathcal{F}_{\varepsilon,\barr{\rho}}(t_{n+1}^-,\flot_N) + H^{n}(t_{n+1}^-) \,.
\end{aligned} 
\end{equation}
Hence we get
\[
E^{n+1}(t_{n+1}) \leq  (E^n(t_n)+ C_6  \varepsilon \tau +  (C_3 + \mathcal{E}_\varepsilon(0,\flot_N) ) \frac{\tau^2}{\varepsilon}  ) \exp(C_5 \tau)\,.
\]

\begin{remark}\label{rem:relentropy} Note that the quantity 
\begin{equation}\label{eq:fplushdef}
 \mathcal{F}_{\varepsilon,\barr{\rho}}(t,\flot_N) + H^{n}(t)\,, \quad \text{for } t\in [t_n, t_{n+1})\,,
\end{equation} 
  can be regarded as a different approximation of the relative internal energy of the continuous setting \eqref{eq:relpot}. Using this quantity instead of simply $\mathcal{F}_{\varepsilon,\barr{\rho}}(t,\flot_N)$ allows us to relate the estimates across different time steps as in equation \eqref{eq:fplush} wihtout having to deal with the discontinuities in time of $\rho_N^\varepsilon$. Nonetheless, the sum in  \eqref{eq:fplushdef} is not positive in general, which is why we define the relative internal energy by $\mathcal{F}_{\varepsilon,\barr{\rho}}(t,\flot_N)$ only.
\end{remark}

\paragraph{\textbf{Step 4: Discrete Gr\"onwall's argument}}
Since $t_{N_T} = \tau N_T =  T$, we obtain 
\[
E^{N_T}(T) \leq E^0(0) \exp(C_5T)+ (C_6 \varepsilon  +  (C_3 + \mathcal{E}_\varepsilon(0,\flot_N) ) \frac{\tau}{\varepsilon}  ) \frac{\exp(C_5 (T+\tau)) -1}{C_5}\,.
\]
Using once again equation \eqref{eq:Hn}, this implies

\[
\begin{aligned}
\mathcal{E}_{\rho,u}(T,\flot_N) \leq & (\mathcal{E}_{\rho,u}(0,\flot_N) + H^0(0))\exp(C_5T) \\&+ (C_6  \varepsilon  + (C_3 + \mathcal{E}_\varepsilon(0,\flot_N) ) \frac{\tau}{\varepsilon}  ) \frac{\exp(C_5 (T+\tau)) -1}{C_5}\\ 
& \quad + \frac{1}{2} \mathcal{E}_{\rho,u}(T,\flot_N)  + C_4 \varepsilon\,.
\end{aligned}
\]
Hence, we get 
\begin{equation}\label{eq:finalestimate}\begin{aligned}
\mathcal{E}_{\rho,u}(T,\flot_N) \leq &  2 (\mathcal{E}_{\rho,u}(0,\flot_N) + H^0(0))\exp(C_5T) \\&+ 2(C_6  \varepsilon  + (C_3 + \mathcal{E}_\varepsilon(0,\flot_N) ) \frac{\tau}{\varepsilon}  ) \frac{\exp(C_5 (T+\tau)) -1}{C_5}+2 C_4 \varepsilon\,.\\ 
\end{aligned}
\end{equation}

In order to conclude the proof we need to estimate the initial energy $\mathcal{E}_\varepsilon(0,\flot_N)$ and the quantity $\mathcal{E}_{\rho,u}(0,\flot_N)  + H^0(0)$.  Note that, due to the initial conditions \eqref{eq:disceuler}
\[
\begin{aligned}
\mathcal{E}_\varepsilon(0,\flot_N) & = \sum_{i=1}^N \frac{1}{2} |\dot{\flot}_N^i(0)|^2 \rho_0[P_i] + \mathcal{F}_\varepsilon(\flot_N(0))\\
& \leq \frac{1}{2}\| u(0) \|^2_{L^\infty(M)}+ \mathcal{U}(\rho(0)) + \frac{W^2_2(\rho(0),\rho_N(0))}{2\varepsilon}\\
& \leq C_3 + \mathcal{U}(\rho(0)) + \frac{\delta_N^2}{2\varepsilon}
\,,
\end{aligned}
\]
where  $\delta_N$ is the error in the initial conditions in the Wasserstein distance, i.e. 
\begin{equation}\label{eq:deltaN}
\delta_N \coloneqq W_2(\rho_N(0),\barr{\rho}(0))\,.
\end{equation}
%

In order to bound the quantity $\mathcal{E}_{\rho,u}(0,\flot_N)  + H^0(0)$, we first estimate the term
\begin{multline*}
\mathcal{F}_{\rho}(0,\flot_N)  + H^0(0) = \frac{W^2_2(\rho_N(0),{\rho}_N^\varepsilon(0))}{2 \varepsilon} \\+ \mathcal{U}(\rho^\varepsilon_N(0)) -\mathcal{U}(\barr{\rho}(0)) - \int_{\mathbb{R}^d} U'(\tilde{\rho}(0)) \ed (\rho_N(0) - \textbf{1}_M \barr{\rho}(0))\,.
\end{multline*}
By definition of $\rho^\varepsilon_N(0)$ we find 
\[
\frac{W^2_2(\rho_N(0),{\rho}_N^\varepsilon(0))}{2 \varepsilon} + \mathcal{U}(\rho^\varepsilon_N(0)) -\mathcal{U}(\barr{\rho}(0)) \leq \frac{W^2_2(\rho_N(0),\barr{\rho}(0))}{2 \varepsilon}\,.
\]
Moreover,
\[
\begin{aligned}
\left|\int_{\mathbb{R}^d} U'(\tilde{\rho}(0)) \ed (\rho_N(0) - \textbf{1}_M \barr{\rho}(0)) \right|   &\leq \mathrm{Lip}(U'(\barr{\rho}(0))) W_1(\rho_N(0) , \barr{\rho}(0))\\
& \leq \frac{C_0^2}{2}  \varepsilon + \frac{W_2^2(\rho_N , {\rho}_N(0))}{2\varepsilon},
\end{aligned}
\]
where $C_0 \coloneqq \mathrm{Lip}(U'(\barr{\rho}(0)))$. 
Combining the two estimates and recalling \eqref{eq:deltaN} we get
\begin{equation}\label{eq:relativeenergy0}
\mathcal{F}_{\rho}(0,\flot_N)  + H^0(0) \leq \frac{C_0^2}{2}  \varepsilon +  \frac{W^2_2(\rho_N(0),\barr{\rho}(0))}{ \varepsilon} =   \frac{C_0^2}{2}  \varepsilon +  \frac{\delta_N^2}{ \varepsilon} \,.
\end{equation}
The remaining terms in the relative energy $\mathcal{E}_{\barr{\rho},u}(0,\flot_N)$ can be estimated by the fact that $\mathcal{K}(0,\flot_N) =0$ (due to the initial conditions \eqref{eq:disceuler}) and the bound
\begin{equation}\label{eq:projectionbound}
\delta_N = W_2(\rho_N(0),\barr{\rho}(0)) = {\|P_{\mathbb{X}_N}\mathrm{Id} - \mathrm{Id}\|} \leq \sqrt{\rho_0[M]} h_N\,,
\end{equation}
which follows from definition of $h_N$.

We conclude by replacing the estimates above into equation \eqref{eq:finalestimate} and estimating the constants using Lemma \ref{lem:continuityextension}.
\end{proof}

We now turn to the proof of Theorem \ref{th:convsemigradient}. We will focus only on the differences with the proof of Theorem \ref{th:convsemieuler}. In particular the kinetic energy will not be taken into account in the definition of the energy.

\begin{proof}[Proof of Theorem \ref{th:convsemigradient}] 
The proof follows the same line as the one of Theorem \ref{th:convsemieuler}. We denote by $\tilde{\rho}$ and $\tilde{u}$ the extensions of $\rho$ and $\barr{u} \coloneqq -\nabla U'(\barr{\rho})$ constructed via Lemma \ref{lem:continuityextension}. In particular, note that $\tilde{u}\neq-\nabla U'(\tilde{\rho})$ outside the domain. In the case where $|U'''_+(0)|<\infty$, we also extend $U$ as a $C^3$ function on $\mathbb{R}$ as in the proof of Theorem \ref{th:convsemieuler}.

Then we take as relative energy 
\begin{equation}\label{eq:modenergygradient}
\mathcal{Z}_{\barr{\rho}}(t,\flot_N) \coloneqq \mathcal{F}_{\varepsilon,\barr{\rho}}(t,\flot_N) + \frac{1}{2}\| \flot_N(t) - \flot(t) \|^2\,.
\end{equation}
By equation \eqref{eq:derpotential}, the time derivative of  $\mathcal{Z}_{\barr{\rho},\barr{u}}(t,\flot_N)$ satisfies
\begin{equation}\label{eq:dertotalgrad0}
\frac{\ed}{\ed t} \mathcal{Z}_{\rho}(t,\flot_N) + \frac{\ed}{\ed t} H^n(t) = \sum_{i=4}^8 J_i -\langle  \dot{\flot}_N(t)+\nabla U'(\tilde{\rho}(t, \flot_N(t))), \dot{\flot}_N(t) - \tilde{u}(t,\flot_N(t)) \rangle\,,
\end{equation}
where the terms $H^n(t)$ and $J_i$ are defined as in equation \eqref{eq:Hndef} and \eqref{eq:dertotal}, respectively. Adding and subtracting $\tilde{u}(t,X_N(t))$ and $\nabla U(\tilde{\rho}(t,X(t))$ in the last term we obtain
\begin{equation}\label{eq:dertotalgrad}
\frac{\ed}{\ed t} \mathcal{Z}_{\rho}(t,\flot_N) + \frac{\ed}{\ed t} H^n(t) +2 \mathcal{K}_{\tilde{u}}(t,\flot_N) = \sum_{i=4}^{10} J_i \,,
\end{equation}
where
\[
J_9:= \langle  \tilde{u}(t,X(t)) - \tilde{u}(t,X_N(t)), \dot{\flot}_N(t) - \tilde{u}(t,\flot_N(t)) \rangle \,,
\]
\[
J_{10}:= \langle  \nabla U (\tilde{\rho}(t,X(t))) -\nabla U (\tilde{\rho}(t,X_N(t))), \dot{\flot}_N(t) - \tilde{u}(t,\flot_N(t)) \rangle \,.
\]

 The estimates for the terms $J_i$ are analogous to those in the proof of Theorem \ref{th:convsemieuler}. In particular, we obtain
\[
\sum_{i=4}^6 J_i \leq \mathcal{K}_{\varepsilon, \tilde{u}}(t,\flot_N) + C_1 \mathcal{Z}_{\barr{\rho}}(t,\flot_N) \,,
\]
where now $C_1 \coloneqq 1+A' \mathrm{Lip}_T(\tilde{u})$, and as in the previous proof $A' \coloneqq \max(2,A)$. The terms $J_7$ and $H^n$ are estimated as in equations \eqref{eq:J7} and \eqref{eq:Hn}, respectively, with the same constants $C_2$ and $C_4$. For $J_8$, proceeding as in \eqref{eq:J8bound}, we obtain
\begin{equation}\label{eq:J8grad}
\begin{aligned}
J_8 & \leq  \frac{1}{\varepsilon} \int_{t_n}^t \left(\frac{1}{2}\| \dot{\flot}_N(t') \|^2 + \frac{1}{2}\| u(t,\flot_N^\varepsilon(t_n) )\|^2 \right)\, \ed t' \\
    & \leq \frac{1}{2\varepsilon}(\frac{\|\flot_N(t_n)-\flot_N^\varepsilon(t_n)\|^2}{2\varepsilon}-\frac{\|\flot_N(t_{n+1})-\flot_N^\varepsilon(t_n)\|^2}{2\varepsilon})+\frac{\tau}{2\varepsilon} \| u\|^2_{L^\infty([0,T]\times M)}\\
   & \leq \frac{1}{2\varepsilon}\left(\mathcal{F}_\varepsilon(\flot_N(t_n))-\mathcal{F}_\varepsilon(\flot_N(t_{n+1}))\right)+\frac{\tau}{2\varepsilon} \| u\|^2_{L^\infty([0,T]\times M)}\\
 & \eqqcolon \frac{\Delta^n}{2\varepsilon}  +  C_3 \frac{\tau}{\varepsilon} \,,
\end{aligned}
\end{equation}
where we used the equation
 \[
\|\dot{\flot}_N(t)\|^2 = -\frac{\ed}{\ed t} \frac{\|\flot_N(t)-\flot_N^\varepsilon(t_n)\|^2}{2\varepsilon}
\] 
to pass from the first to the second line, and the inequality 
\[
\frac{\|\flot_N(t_{n+1})-\flot_N^\varepsilon(t_n)\|^2}{2\varepsilon} + \mc{U}(\rho_N^\varepsilon(t_n)) \geq \mathcal{F}_\varepsilon(\flot_N(t_{n+1}))
\]
 to pass from the second to the third line. Finally, the last two terms are estimated as follows
 \[
 \begin{aligned}
 J_9+J_{10}  &\leq \frac{1}{2}\mc{K}_{\varepsilon,\tilde{u}}(t,X_N(t)) + 2(\mathrm{Lip}_T(\tilde{u})+\mathrm{Lip}_T(\nabla U'(\tilde{\rho})) \| X_N(t) - X(t)\|^2\\
 &\eqqcolon \frac{1}{2}\mc{K}_{\varepsilon,\tilde{u}}(t,X_N(t)) + C_5 \mathcal{Z}_\rho(t,X_N)\,.
 \end{aligned}
 \]

Introducing $Z^n(t)\coloneqq  \mathcal{Z}_{\rho}(t,\flot_N) +  H^n(t)$, and proceding as in the previous proof, we obtain 
\[
\begin{aligned}
\frac{\ed}{\ed t} Z^n(t)  +  \frac{1}{2} \mathcal{K}_{\varepsilon, \tilde{u}}(t,\flot_N) & \leq  2 (C_1+C_2 +C_5) Z^n(t) + (\frac{C_2}{2} + 2 (C_1+C_2 +C_5 )C_4) \varepsilon \\ & \quad + \frac{\Delta^n}{2\varepsilon} + C_3 \frac{\tau}{\varepsilon} \\ 
&\eqqcolon  C_6 Z^n(t) + C_7 \varepsilon +C_3 \frac{\tau}{\varepsilon} + \frac{\Delta^n}{2\varepsilon} \,.
\end{aligned}
\]
Therefore, by the same reasoning as above
\[
\begin{aligned}
\mathcal{Z}_{\rho}(T,X) + \frac{1}{2} \int_0^T \mathcal{K}_{\varepsilon, \tilde{u}}(t,\flot_N) \leq & 2 (\mathcal{Z}_{\rho,u}(0,\flot_N)  + H^0(0)  +C_7  \varepsilon  + C_3 \frac{\tau}{\varepsilon} )\exp(C_6 T) \\  &+ \frac{\tau}{\varepsilon} (\mathcal{F}_\varepsilon({\flot_N(0)}) - \mathcal{F}_\varepsilon({\flot_N(T)})) \exp(C_6 T) + 2 C_4 \varepsilon\,.
\end{aligned}
\]
However, note that
\[
\mathcal{F}_\varepsilon({\flot_N(0)}) - \mathcal{F}_\varepsilon({\flot_N(T)}) \leq \mathcal{F}_\varepsilon({\flot_N(0)}) - \min \mc{U} \leq \mathcal{U}(\rho(0)) - \min \mathcal{U} + \frac{W^2(\rho(0),\rho_N(0))}{2\varepsilon} \,.\]
We conclude the proof using \eqref{eq:relativeenergy0} and \eqref{eq:projectionbound} to bound this latter term as well as $\mathcal{Z}_{\rho}(0,\flot_N)  + H^0(0)$, and using Lemma \ref{lem:continuityextension} to bound the constants in the final estimate.
\end{proof}

\begin{remark}\label{rem:deltaN}
Note that the dependency of our estimates on $h_N$ is only due to the bound \eqref{eq:projectionbound}. In particular, the error estimates in Theorem \ref{th:convsemieuler} and   \ref{th:convsemigradient} hold also replacing $h_N$ with $\delta_N$. 
\end{remark}

\begin{remark}
We observe that we can obtain similar convergence estimates also on the Lagrangian velocity as it can be easily verified with the following triangular inequality
\begin{equation}\label{eq:lagvelocity}
\begin{aligned}
\| \dot{\flot}_N(t) - \dot{ \flot}(t)\|_{\mathbb{X}} & \leq  \| \dot{\flot}_N(t) - \tilde{u}(t, \flot_N(t))\|_{\mathbb{X}} + \| \tilde{u}(t, \flot_N(t))- \tilde{u}(t, \flot(t))\|_{\mathbb{X}}\\
 & \leq \sqrt{2 \mathcal{K}_{\tilde{u}}(t,\flot_N)} +\mathrm{Lip}_T(\tilde{u})\|\flot_N(t) - \flot(t)\|_{\mathbb{X}}\,.
 \end{aligned}
\end{equation}
\end{remark}

\begin{remark} The regularity of the exact solutions required in Theorem \ref{th:convsemieuler} and \ref{th:convsemigradient} is chosen in order to apply the extension Lemma \ref{lem:continuityextension}. However, examining the constants appearing in the estimates above, one can see that this is stronger than what is actually required from the extended variables themselves. For example, one can check that the proof still holds if $u$ is of class $C^{1,1}$  on $[0,T]\times M$ with $C^{1,1}$ divergence in space, uniformly in time, and admits an extension $\tilde{u}$ with the same regularity. If $M$ is sufficiently regular, say simply connected  with  a smooth boundary, such an extension can be constructed using Fefferman's extension theorem \cite{fefferman2009extension} as in Lemma \ref{lem:continuityextension} but applied to the potentials obtained via the Helmholtz decomposition of $u$. 
\end{remark}
\section{Implementation}

\subsection{Computation of the Moreau-Yosida regularization}
In this section we describe how the schemes \eqref{eq:disceulertime} and \eqref{eq:discgradienttime} can be implemented. In particular, we show that computing  the gradient vector field driving the dynamics amounts to solving a semi-discrete optimal transport problem at each time step.

\begin{definition}[Laguerre diagram]
The Laguerre diagram of $(x_1,\ldots,x_N)\in  (\mathbb{R}^d)^N$  with weights $(w_1,\ldots,w_N) \in \mathbb{R}^N$ is a decomposition of $M$ into $N$ subsets  $(L_i)_i$ defined by 
\[
L_i \coloneqq\{ x\in M ~|~ \forall j\in\{1,\ldots,N\},  |x-x_i|^2+w_i  \leq |x-x_j|^2+w_j\}\,.
\]
\end{definition}

In the following we will identify $\mathbb{X}_N$ with $(\mathbb{R}^d)^N$, i.e.\ we regard an element $\flot_N \in \mathbb{X}_N$ as the collection of the particle positions $(\flot_N^i)_i \in (\mathbb{R}^d)^N$. With this identification, the functional $\mathcal{F}_\varepsilon$ can be interepreted as a function on $(\mathbb{R}^d)^N$, and its gradient at a given point as a vector in $(\mathbb{R}^d)^N$. Let us introduce the set
\[
\mathcal{D}_N \coloneqq \{ (x_1,\ldots, x_N) \in (\mathbb{R}^d)^N ~|~ x_i=x_j\, \text{ for some }\, i\neq j\}\,.
\]

In the following proposition we collect the results of Proposition 11 and 13 in \cite{sarrazin2020lagrangian} adapted to our setting. It gives the explicit expression of the regularized density and the gradient of the regularized energy appearing in the time-continuous schemes given by \eqref{eq:disceuler} and \eqref{eq:discgradient}.
\begin{proposition}\label{prop:gradient}
Let $X=(x_1,\ldots,x_N)\in (\mathbb{R}^d)^N\setminus \mc{D}_N$, and set $\rho_N = \sum_i \rho_0[P_i] \delta_{x_i}$. Then the unique minimizer $\rho_N^\varepsilon$ of problem satisfies 
\[
\rho_N^\varepsilon(x) = (2\varepsilon U')^{-1}((w_i-|x-x_i|^2) \vee U'(0)) \,,\quad \forall \,x\in L_i\,,
\]
where $(L_i)_i$ is the Laguerre diagram associated with the positions $(x_1,\ldots,x_N)$ and the weights $(w_1,\ldots,w_N)$, which are uniquely defined up to an additive constant by the condition $\rho_N^\varepsilon[L_i] = \rho_0[P_i]$. Moreover, $\mathcal{F}_\varepsilon$ interpreted as a function on $(\mathbb{R}^d)^N$ is  continuously differentiable on $(\mathbb{R}^d)^N\setminus \mc{D}_N$ and
\[
\nabla_{x_i} \mc{F}_\varepsilon(X) = \rho_0[P_i] \frac{x_i-b_i(X)}{\varepsilon}\,,\quad b_i(X) \coloneqq \frac{1}{\rho_0[P_i]} \int_{L_i} x\rho_N^\varepsilon \ed x\,. \]
\end{proposition} 

\begin{remark}[Power energies] If the energy is defined by the power function
\[
U(r) = \frac{r^m}{m-1}\,,
\]
for $m>1$, then the minimizer $\rho_N^\varepsilon$ has the following form:
\[
\rho_N^\varepsilon(x) = \left[ \left(\frac{m-1}{m }\right) \frac{(w_i - |x-x_i|^2)_+}{2\varepsilon}\right]^{\frac{1}{m-1}} \quad \forall \,x\in L_i\,.
\]
\end{remark}

Actually, in order to compute the solutions of the fully-discrete scheme, we do not need the expression for the  gradient in Proposition \ref{prop:gradient}, but we just need to identify $P_{\mathbb{X}_N} X_N^\varepsilon(t_n)$ in \eqref{eq:disceulertime} and \eqref{eq:discgradienttime}. For this, assume that $(X_N^i(t_n))_i\in (\mathbb{R}^d)^N\setminus \mc{D}_N $ and let $(L_i)_i$ be the Laguerre diagram associated with $\rho_N^\varepsilon(t_n)$. Then, for any $Y_N \in \mathbb{X}_N$, we have
\[
\int_{S_0} X_N^\varepsilon(t_n) \cdot Y_N \rho_0 \, \ed x = \sum_i Y_N^i \cdot \int_{P_i} X_N^\varepsilon(t_n) \rho_0 \,\ed x= \sum_i Y_N^i \cdot \int_{L_i} x \rho_N^\varepsilon(t_n)\,\ed x \,.
\]
Therefore,
\[
P_{\mathbb{X}_N} X_N^\varepsilon(t_n)(\omega) = b_i(X_N(t_n)) \, \quad \forall\, \omega \in P_i\,,
\]
where $b_i(X_N(t_n))\in\mathbb{R}^d$ is the barycenter of $\rho_N^\varepsilon(t_n)$ restricted on $L_i$.

\begin{remark}[Initialization by optimal quantization]\label{rem:quantization} The partition $\mathcal{P}_N$ of the support $S_0\subseteq M$ of $\rho_0$ which is required to define the space $\mathbb{X}_N$ (see Section \ref{sec:spacedisc}) can be itself defined as the interesection of a Laguerre diagram $(L_i)_i$ with $S_0$. For instance, assuming the masses to be equal, i.e.\ $m_i = \rho_0[M]/N$ for $i=1,\ldots,N$, one can select the vector of positions $(x_1,\ldots,x_N)\in (\mathbb{R}^{d})^N$ defining the diagram 
to belong to the $\mathrm{argmin}$ of
\[
(y_1,\ldots,y_N)\in (\mathbb{R}^{d})^N \mapsto W_2\bigg(\sum_i \frac{\rho_0[M]}{N} \delta_{y_i}, \rho_0\bigg)\,,
\]
so that  there exists a vector of weights $(w_1,\ldots,w_N)\in \mathbb{R}^N$ such that $\rho_0[L_i]=\rho_0[M]/N$. 
Then one can define the initial conditions $\flot_N(0)$ by $\flot_N(0)|_{L_i} = x_i$, and therefore $\rho_N(0) = \sum_i  \frac{\rho_0[M]}{N}  \delta_{x_i}$. With this choice $\delta_N =W_2(\rho_N(0),\rho_0) \lesssim N^{-1/d}$ (see, e.g., \cite{kloeckner2012approximation}). In view of Remark \ref{rem:deltaN}, this ensures the convergence of the schemes independently of the size of the partion $h_N$.
\end{remark}

\subsection{Time integration and linear potentials} The schemes \eqref{eq:disceuler} and \eqref{eq:discgradient} can be easily generalized to the case when the energy of the system contains an additional linear term of the form 
\[
\int_M V\, \ed \rho\,,
\]
where $V \in C^{1,1}(M)$ is a given function. At the discrete level, it is more convenient to treat this term independently of the Moreau-Yosida regularization, i.e.\ by adding to the discrete energy the term
\begin{equation}\label{eq:linearpot}
\int_{\mathbb{R}^d} \tilde{V}\, \ed \rho_N = \int_M \tilde{V} \circ \flot_N \, \ed \rho_0\,,
\end{equation}
where $\tilde{V}$ is a $C^{1,1}$ extension of $V$, e.g., constructed using Fefferman's extension theorem \cite{fefferman2009extension}. Then, in view of Proposition \ref{prop:gradient} the discrete scheme \eqref{eq:disceuler} would be replaced by
\begin{equation}\label{eq:eulerlinearpot}
\ddot{\flot}^i_N(t) = - \frac{{\flot}^i_N(t)-b_i( {\flot}_N(t_n))}{\varepsilon}-\nabla \tilde{V}(\flot_N^i(t))\,,
\end{equation}
for all $i\in\{1,\ldots,N\}$ and $t\in[t_n,t_{n+1})$, where $b_i$ is defined as in Proposition \ref{prop:gradient}. Therefore for each time-step one needs to:
\begin{enumerate}
\item find the optimal density $\rho_N^\varepsilon(t_n)$ and the associated barycenters $b_i( {\flot}_N(t_n))$: as in \cite{kitagawa2019convergence}, this is done by applying a damped Newton's method to solve the system of optimality conditions  $\rho_N^\varepsilon(t_n)[L_i] = \rho_0[P_i]$ from Proposition \ref{prop:gradient};
\item solve $N$ decoupled systems of ODEs in \eqref{eq:eulerlinearpot}, which can be done explicitly for particular choices of $\tilde{V}$. 
\end{enumerate}
The same holds for the scheme   \eqref{eq:discgradient} upon replacing $\ddot{\flot}^i_N(t)$ by $\dot{\flot}^i_N(t)$.

Finally, note  that even with the additional term \eqref{eq:linearpot}, the proofs of convergence above still apply without modifying the relative energies and with only minor modifications. In particular, the constant in Theorem \ref{th:convsemieuler} and \ref{th:convsemigradient} would additionally depend on $\mathrm{Lip}(\nabla \tilde{V})$.

\section{Numerical tests}

In this section we demonstrate numerically the behavior of the scheme  in terms of convergence with mesh and time-step refinement. The tests presented hereafter correspond to the internal energy/pressure function 
\begin{equation}\label{eq:quadratic}
U(r) = P(r) = r^2 \,,
\end{equation}
for which the Euler equations \eqref{eq:euler} yield the shallow water equations without rotation and the gradient flow \eqref{eq:porous} yields the porous medium equation with a quadratic nonlinearity. Note, however, that in tests below the vector field $\nabla U'(\rho)$ is not Lipschitz (in fact, it is discontinuous at the boundary of the support of $\rho$), so they are outside the limits of applicability of our theorems. 
For all the tests the discrete initial condition are determined by optimal quantization with respect to the Wasserstein distance as in Remark \ref{rem:quantization}. 

For the computation of the Moreau-Yosida regularization, we used the open-source library \texttt{sd-ot}, which is available at \url{https://github.com/sd-ot}.

\subsection{Convergence: porous medium equation}
The porous medium equation \eqref{eq:porous} associated with the energy \eqref{eq:quadratic} admits the following exact solution
\begin{equation}\label{eq:baren}
\rho(t,x) = \frac{1}{\sqrt{t}} \left (C^2 - \frac{1}{16\sqrt{t}} |x|^2 \right)_+
\end{equation}
on any time interval $[t_0,T]$, with $t_0>0$.
Initial conditions are given by optimal quantization of the Barenblatt profile at given time. Equation \eqref{eq:baren} describes the evolution of the so-called Barenblatt profile. The internal energy decays according to
\[
\mathcal{U}(t) = \frac{16 \pi C^6}{3 \sqrt{t}}\,,
\]
whereas the Lagrangian flow is given by
\begin{equation}
\flot(t,x) = x \left (\frac{t}{t_0} \right)^{1/4}\,.
\end{equation}
Here we take $t_0=1/16$, $T=1$ and $C=1/3$, and we monitor the following quantities:
\begin{equation}\label{eq:errorflow}
\Delta \flot \coloneqq \| \flot_N(T) - \flot(T,\flot_N(0)) \|_{\mathbb{X}} \,,\quad \Delta \mathcal{U}\coloneqq | \mathcal{U}_\varepsilon(\rho_N(T))-\mathcal{U}(T) |\,. 
\end{equation}
Note that $\Delta \flot$ is an order one approximation of the $L^2$ distance between the discrete and continuous flows appearing in the convergence estimates.
For a given number of particles $N$, we take $\varepsilon = \sqrt{\tau} = 1/\sqrt{N}$, which implies a rate of convergence of $1/2$ according to Theorem \ref{th:convsemigradient} (see also Remark \ref{rem:quantization}). In Figure \ref{fig:baren} we show the density $\rho_N^\varepsilon$ for fixed $N$ and at different times and the associated Laguerre diagram. Table \ref{tab:baren} collects the errors and the associated convergence rates which confirm our estimate. Figure \ref{fig:barenenergy} shows the time evolution of the internal energy, which decreases monotonically in accordance with the stability estimate \eqref{eq:energystabgrad}.

\begin{figure}
    \begin{center}
        \includegraphics[trim = 350 100 350 0, clip = true, scale=.2]{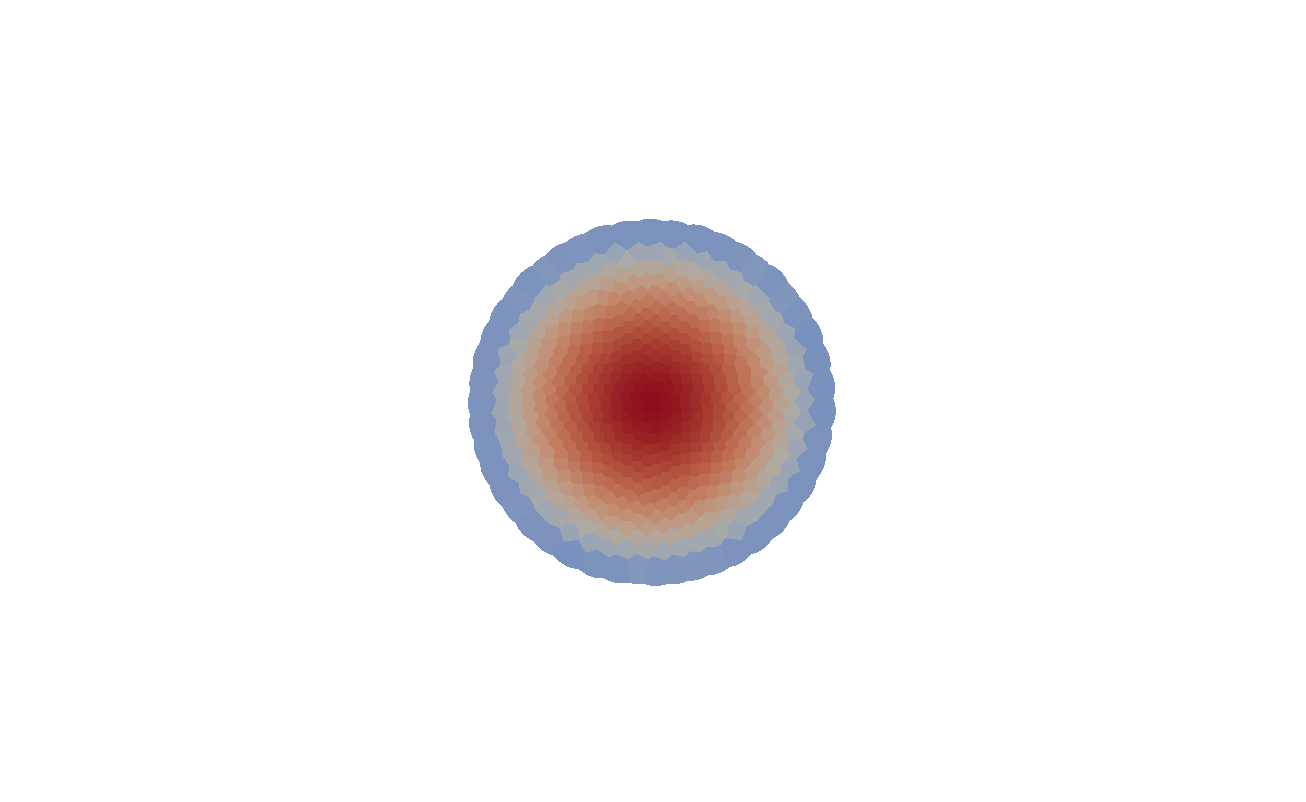}
        \includegraphics[trim = 350 100 350 0, clip = true, scale=.2]{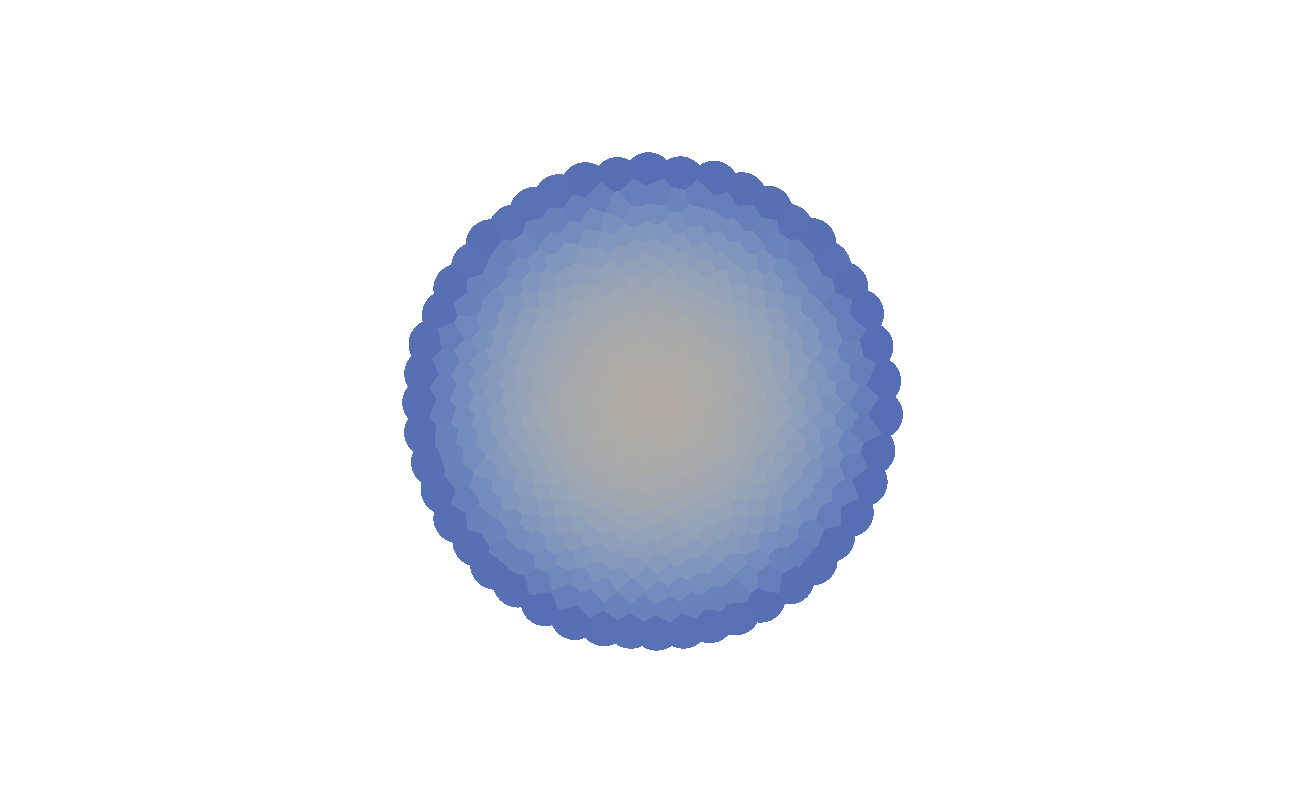}
        \includegraphics[trim = 350 100 150 0, clip = true, scale=.2]{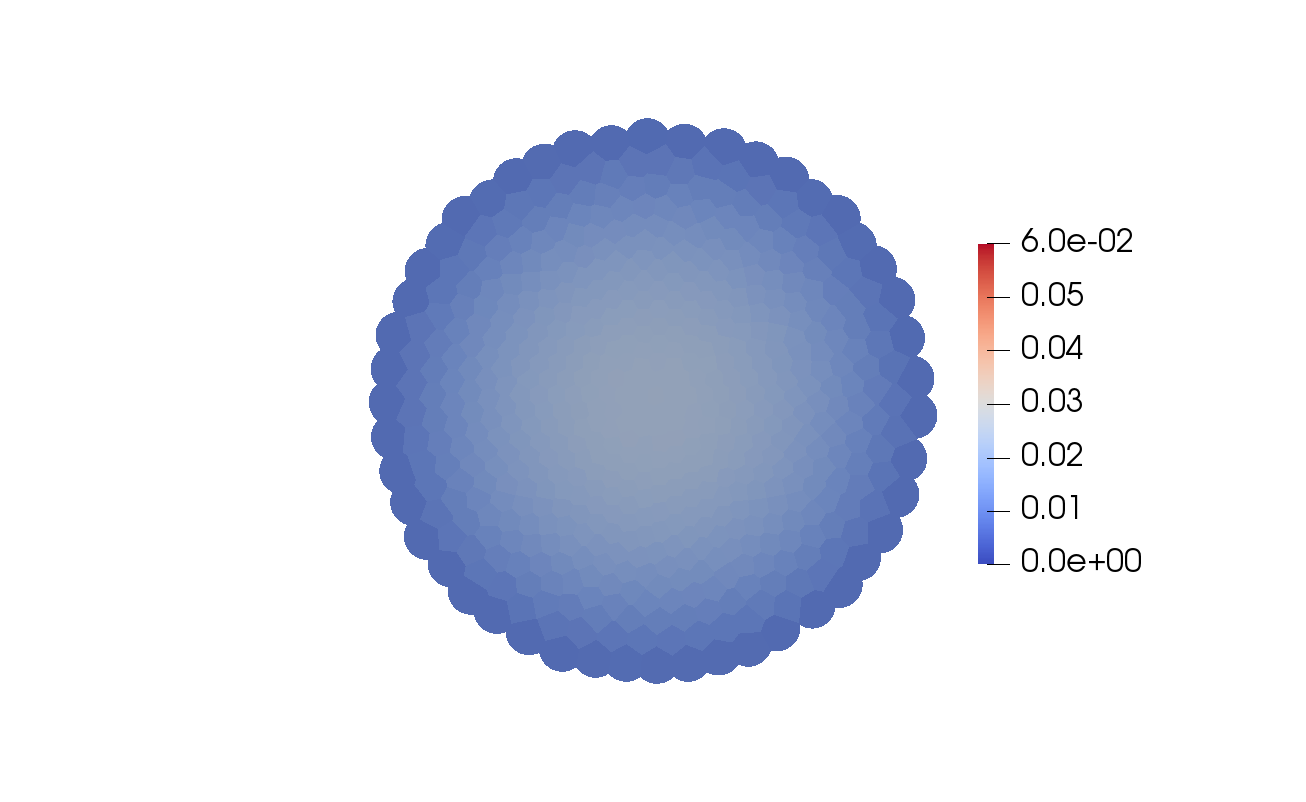}
        
        \includegraphics[trim = 350 0 350 60, clip = true, scale=.2]{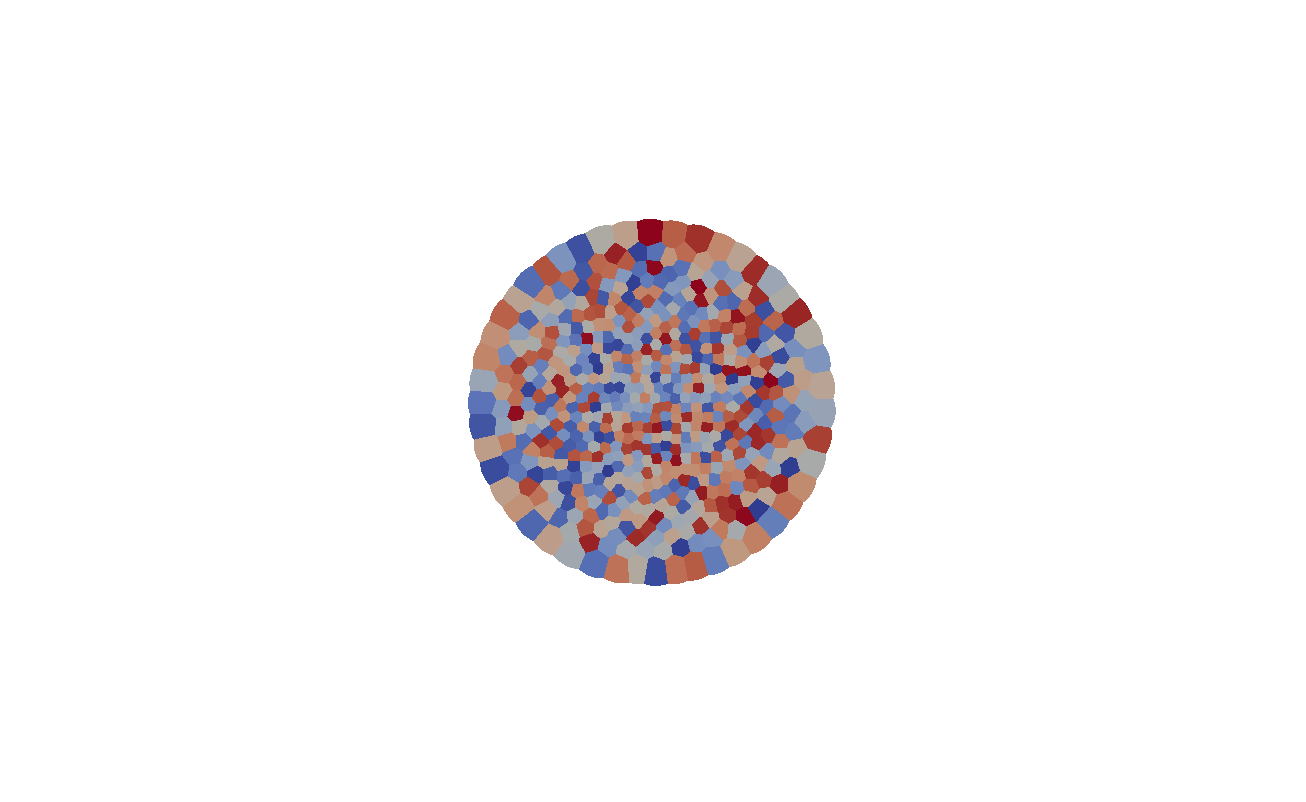}
        \includegraphics[trim = 350 0 350 60, clip = true, scale=.2]{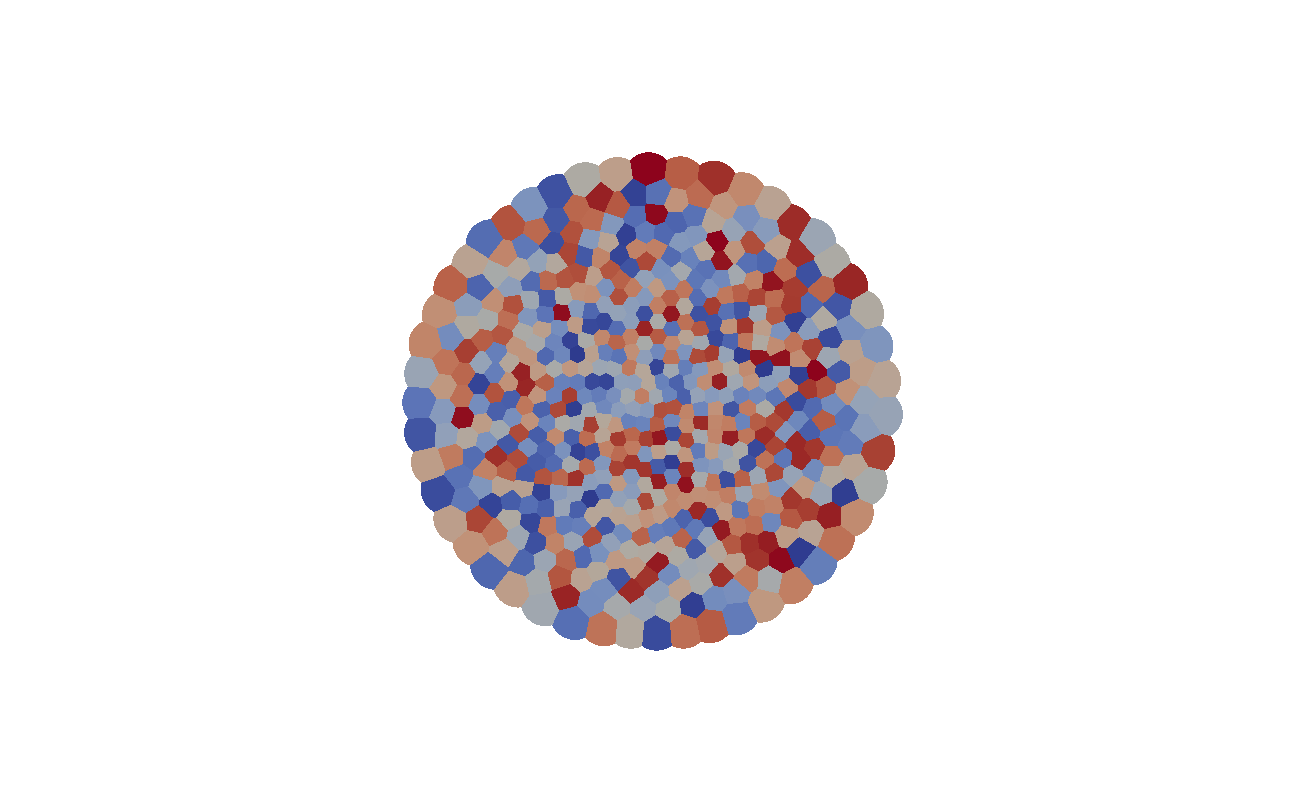}
        \includegraphics[trim = 350 0 150 60, clip = true, scale=.2]{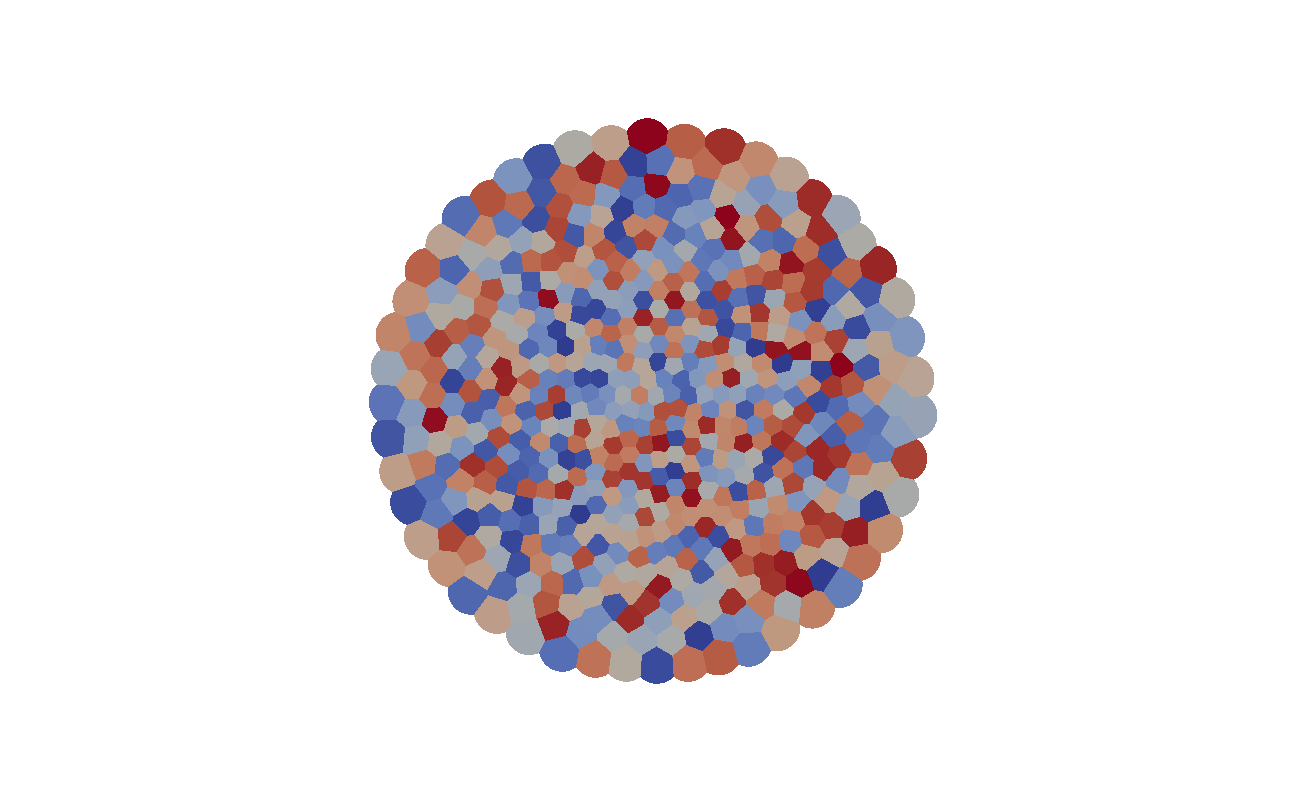}
    \end{center}
    \vspace{-2em}
    \caption{Evolution of the density $\rho_N^\varepsilon$ for the Barenblatt solution of the porous medium equation for $N=576$, and $\varepsilon = \sqrt{\tau} = 1/\sqrt{N}$. Upper row: weights evolution; lower row: Laguerre diagram evolution.}
    \label{fig:baren}
\end{figure}

\begin{table}
\begin{tabular}{lllll}
\hline
 $1/\sqrt{N}$   & $\Delta X$   & rate     & $\Delta {\mathcal{U}}$   & rate     \\
\hline
 1.25e-01       & 4.71e-02           & - & 1.66e-02                 & - \\
 6.25e-02       & 2.78e-02           & 7.62e-01 & 9.39e-03                 & 8.21e-01 \\
 3.12e-02       & 1.55e-02           & 8.44e-01 & 5.11e-03                 & 8.77e-01 \\
 1.56e-02       & 8.24e-03           & 9.08e-01 & 2.72e-03                 & 9.12e-01 \\
\hline
\end{tabular}
\vspace{.5em}

\caption{Errors and convegence rates for the Barenblatt solution of the porous medium equation, with $\varepsilon = \sqrt{\tau} = 1/\sqrt{N}$. }
\label{tab:baren}
\end{table}

\begin{figure}
    \begin{center}
        \includegraphics[scale=.8]{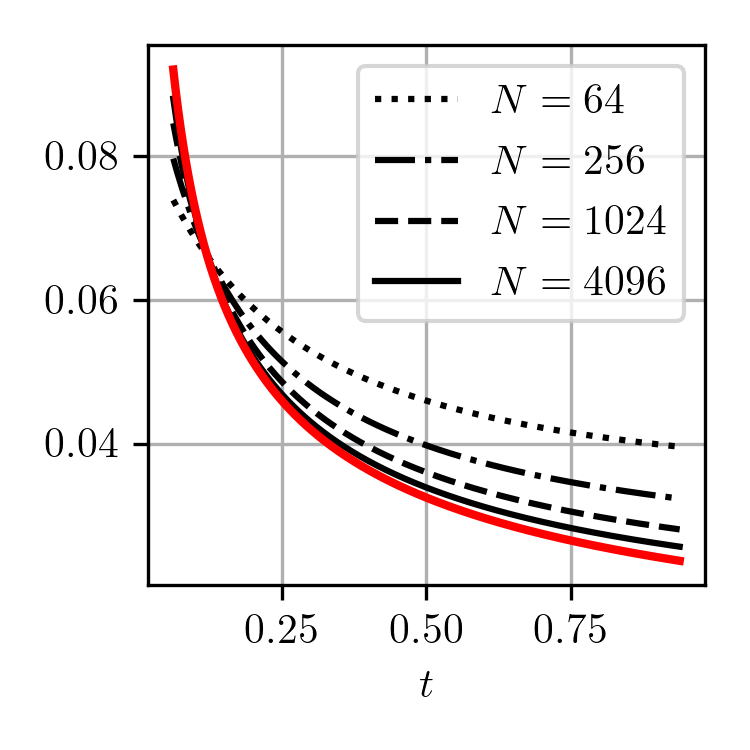}
    \end{center}
    \vspace{-2em}
    \caption{Time evolution of the discrete internal energy $\mc{U}_\varepsilon(\rho_N(t))$ for the Barenblatt solution of the porous medium equation (the red line corresponds to the exact energy evolution).}
        \label{fig:barenenergy}
\end{figure}

\subsection{Convergence: Euler equation}
We perform two different convergence tests for the Euler model \eqref{eq:euler}. For the first we construct an exact solution of the equation by a time rescaling of the Barenblatt solution above,  i.e.\ we take
\begin{equation}
\rho(t,x) = \frac{4}{1 + 2 t + 5 t^2} \left (C^2 - \frac{1}{4(1 + 2 t + 5 t^2)} |x|^2 \right)_+\,.
\end{equation}
This is an exact solution of the model associated with the Lagrangian flow
\begin{equation}
\flot(t,x) = x \sqrt{ 5 t^2 +2 t+ 1}
\end{equation}
and the initial conditions $\dot{\flot}(0,x) =x$.
In this case the exact kinetic and internal energy evolutions are given by
\begin{equation}
\mathcal{K}(t) = \frac{4\pi C^6(10 t+2)^2}{3(5t^2+2t+1)} \,, \quad \mathcal{U}(t) = \frac{64 \pi C^6}{3 (1 + 2 t + 5 t^2)}\,.
\end{equation}
For this test, we take $t_0=0$, $T=0.6$ and $C=1/3$, and we monitor the flow error $\Delta \flot$ defined in equation \eqref{eq:errorflow} and the total energy error
\begin{equation}
\Delta \mathcal{E}\coloneqq | \mathcal{E}_\varepsilon(T,\flot_N)-\mathcal{E}(T) | \,,
\end{equation}
where $\mathcal{E}(T) = \mathcal{K}(T)+ \mathcal{U}(T)$.

For the second test, we add to the system a linear confinement potential
\begin{equation}
V(x) = \frac{5}{8}|x|^2\,.
\end{equation}
Then, we consider the exact solutions associated with the steady density 
\begin{equation}
\rho(x) =  \left (C^2 - \frac{1}{16} |x|^2 \right)_+
\end{equation}
and the rigid rotation given by the flow
\begin{equation}
\flot(t,x) =  R(t) x \,, \quad R(t) =\left( \begin{array}{cc} \cos(t) & \sin(t) \\ -\sin(t)& \cos(t) \end{array}  \right) \,.
\end{equation}
Both the kinetic and internal energy are constant during the evolution and they are given by
\begin{equation}
\mathcal{K}(t) = \frac{64\pi C^6 }{3} \,, \quad \mathcal{U}(t) = \frac{16 \pi C^6}{3}\,.
\end{equation}
For this test we take $t_0=0$, $T=1$ and $C=1/3$, and  we monitor the same quantities as above. 

As before, for a given number of particles $N$, we take $\varepsilon = \sqrt{\tau} = 1/\sqrt{N}$, which implies a rate of convergence of $1/2$ according to Theorem \ref{th:convsemieuler} (see also Remark \ref{rem:quantization}).  Tables \ref{tab:bareneuler} and \ref{tab:rotation} collect the errors and the associated convergence rates for the two tests and confirm our error estimate. Figures \ref{fig:bareneulerenergy} and \ref{fig:rotationenergy} show the time evolution of the total, kinetic and internal energy; note that the discrete total energy decreases monotonically in accordance with the stability estimate \eqref{eq:energystab}.

\begin{table}
\begin{tabular}{lllll}
\hline
 $1/\sqrt{N}$   & $\Delta X$   & rate     & $\Delta{\mathcal{E}}$   & rate     \\
\hline
 1.25e-01       & 4.36e-02      & - & 2.46e-02            & - \\
 6.25e-02       & 2.77e-02      & 6.53e-01 & 1.68e-02            & 5.52e-01 \\
 3.12e-02       & 1.61e-02      & 7.83e-01 & 1.02e-02            & 7.13e-01 \\
 1.56e-02       & 8.80e-03      & 8.71e-01 & 5.71e-03            & 8.44e-01 \\
\hline
\end{tabular}

\vspace{.5em}

\caption{Errors and convegence rates for the Barenblatt solution of the Euler equations, with $\varepsilon = \sqrt{\tau} = 1/\sqrt{N}$.}
    \label{tab:bareneuler}
\end{table}

\begin{figure}
    \begin{center}
         \hspace{-.15em}\includegraphics[scale=.8]{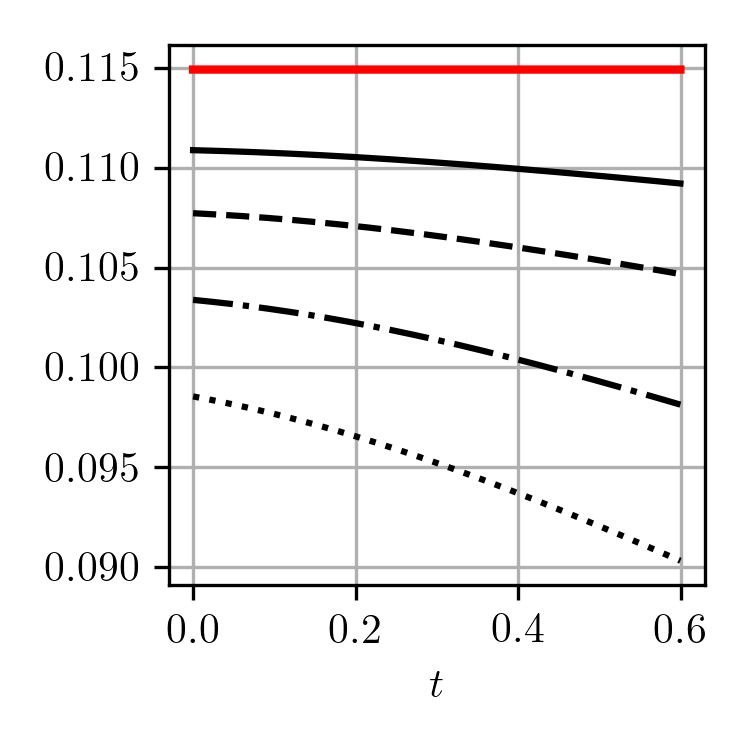}\hspace{-1.09em}
     \includegraphics[scale=.8]{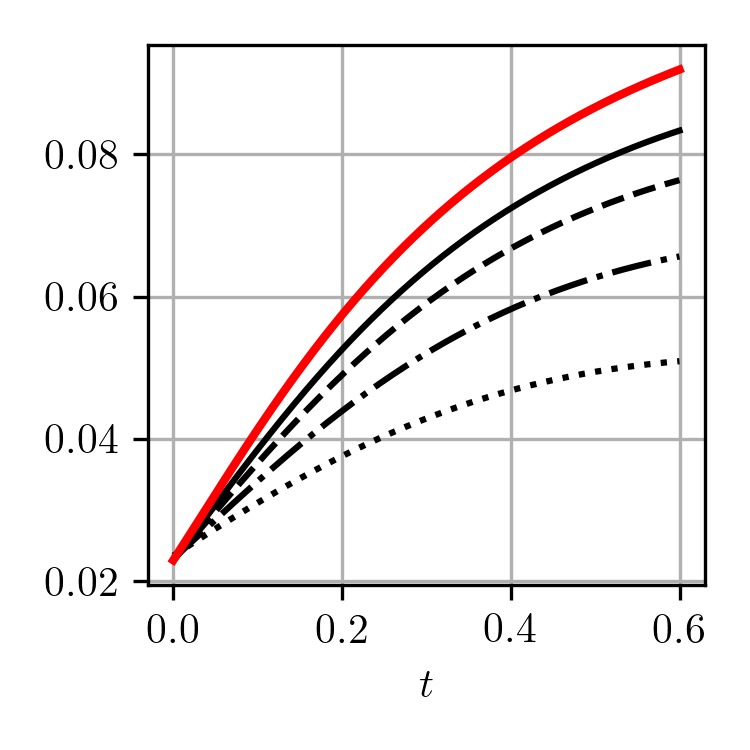}\hspace{-1.09em}
        \includegraphics[scale=.8]{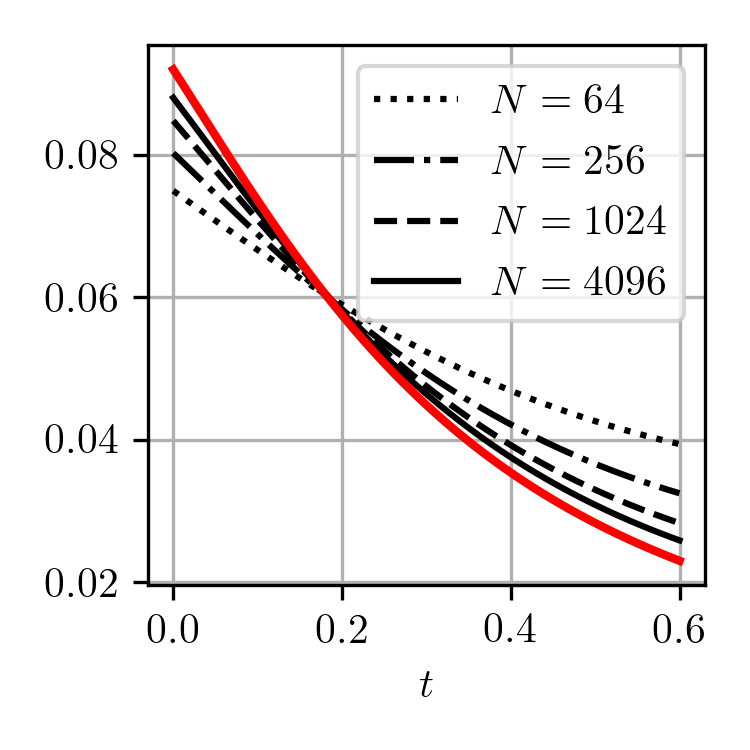}\hspace{-.15em}
    \end{center}
    \vspace{-2em}
    \caption{Time evolution of the discrete total energy $\mathcal{E}_\varepsilon(t,\flot_N)$ (left), kinetic energy (center), and internal energy $\mc{U}_\varepsilon(\rho_N(t))$ (right) for the Barenblatt solution of the Euler equations (the red line corresponds to the exact energy evolution).}
        \label{fig:bareneulerenergy}
\end{figure}

\begin{table}
\begin{tabular}{lllll}
\hline
 $1/\sqrt{N}$   & $\Delta X$   & rate     & $\Delta{\mathcal{E}}$   & rate     \\
\hline
 1.25e-01       & 7.28e-02      & - & 3.00e-02            &- \\
 6.25e-02       & 3.76e-02      & 9.55e-01 & 1.59e-02            & 9.18e-01 \\
 3.12e-02       & 1.92e-02      & 9.71e-01 & 8.16e-03            & 9.61e-01 \\
 1.56e-02       & 9.84e-03      & 9.61e-01 & 4.28e-03            & 9.29e-01 \\
\hline
\end{tabular}
\vspace{.5em}
\caption{Errors and convegence rates for the rigid rotation solution of the Equation equation, with $\varepsilon = \sqrt{\tau} = 1/\sqrt{N}$.}
    \label{tab:rotation}
\end{table}

\begin{figure}
    \begin{center}\hspace{-.5em}
         \includegraphics[scale=.8]{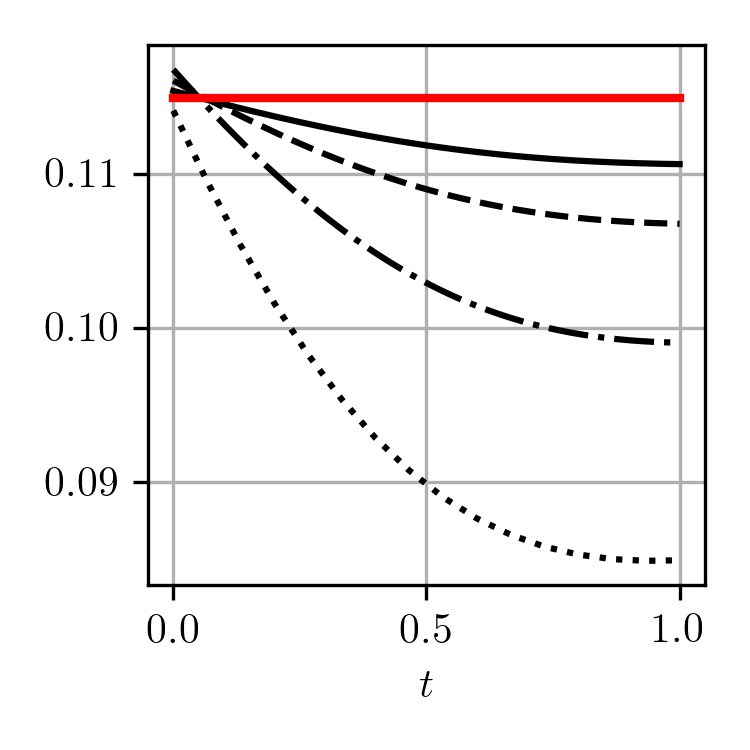}\hspace{-1.09em}
     \includegraphics[scale=.8]{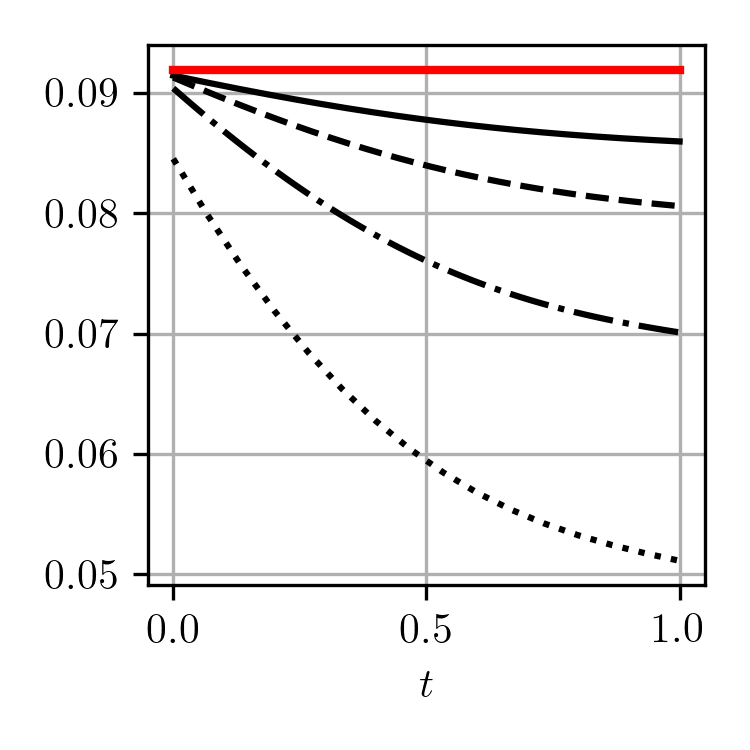}\hspace{-1.09em}
        \includegraphics[scale=.8]{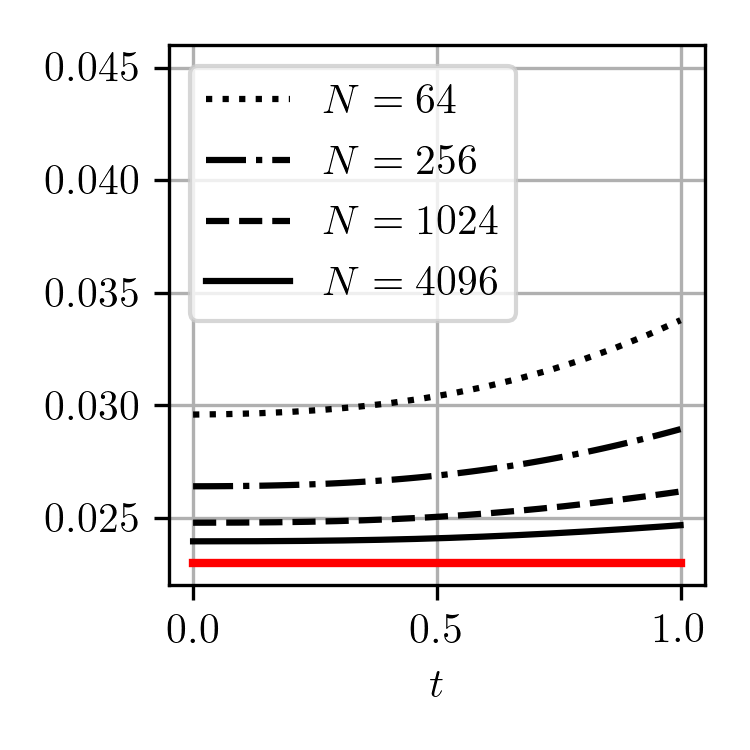}\hspace{-.5em}
    \end{center}
    \vspace{-2em}
    \caption{Time evolution of the discrete total energy $\mathcal{E}_\varepsilon(t,\flot_N)$ (left), kinetic energy (center), and internal energy $\mc{U}_\varepsilon(\rho_N(t))$ (right) for the rigid rotation solution of the Euler equations (the red line corresponds to the exact energy evolution).}    \label{fig:rotationenergy}
\end{figure}

\section*{Acknowledgements}
This work was supported by a public grant as part of the Investissement d'avenir project, reference ANR-11-LABX-0056-LMH, LabEx LMH, and by a grant from the French ANR (MAGA, ANR-16-CE40-0014).

\clearpage

\bibliographystyle{plain}      
\bibliography{refs}   

\end{document}